\RequirePackage[table]{xcolor}
\documentclass[journal,twoside,web]{ieeecolor}
\usepackage{generic}

\usepackage{amsthm}  
\usepackage[smallerops]{newtxmath}
\usepackage[cal=cm,scr=boondoxo,bb=boondox,frak=euler]{mathalfa}
\let\labelindent\relax
\usepackage[inline]{enumitem}
\usepackage[utf8]{inputenc}
\usepackage{mathtools}
\usepackage{bm}
\usepackage{upgreek}
\usepackage{microtype}
\usepackage[subtle,mathspacing=normal]{savetrees}
\usepackage{tikz}
\usepackage{balance}
\usepackage{breqn}
\usepackage[ruled]{algorithm2e}
\usepackage{float}
\usepackage{comment}
\usepackage[style=ieee,backend=biber,doi=false,isbn=false]{biblatex}
\allowdisplaybreaks[1]

\definecolor{NU100}{RGB}{78,42,132}
\definecolor{NU90}{RGB}{91,59,140}
\definecolor{NU80}{RGB}{104,76,150}
\definecolor{NU70}{RGB}{118,93,160}
\definecolor{NU60}{RGB}{131,110,170}
\definecolor{NU50}{RGB}{147,128,182}
\definecolor{NU40}{RGB}{164,149,195}
\definecolor{NU30}{RGB}{182,172,209}
\definecolor{NU20}{RGB}{204,196,223}
\definecolor{NU10}{RGB}{228,224,238}
\definecolor{NUblue1}{RGB}{37,90,124}
\definecolor{NUblue2}{RGB}{61,122,161}
\definecolor{NUblue3}{RGB}{106,179,227}
\definecolor{NUyellow1}{RGB}{194,179,49}
\definecolor{NUyellow2}{RGB}{253,236,86}
\definecolor{NUyellow3}{RGB}{255,240,110}
\definecolor{NUred1}{RGB}{171,43,86}
\definecolor{NUred2}{RGB}{223,76,125}
\definecolor{NUred3}{RGB}{246,106,153}
\definecolor{gray80}{gray}{0.80}
\definecolor{gray90}{gray}{0.90}

\usetikzlibrary{calc,shapes,arrows,positioning,matrix}
\tikzset{block/.style={%
		text height=1.6ex,text depth=0.25ex,
		rectangle,
		minimum size=12mm,inner xsep=4mm,inner ysep=2mm,
		very thick,draw
}}
\tikzset{Lap/.style={%
		text height=1.6ex,text depth=0.25ex,
		rectangle,
		minimum size=12mm,inner xsep=4mm,inner ysep=2mm,
		very thick,draw
}}
\tikzset{sum/.style={%
		circle,
		minimum size=1mm,inner xsep=1mm,inner ysep=1mm,
		very thick,draw}}
\tikzset{link/.style={->,very thick,>=stealth,rounded corners}}
\tikzset{vertex/.style={draw,circle,very thick,black,fill=NUblue3!80}}
\tikzset{colvertex/.style={draw,circle,very thick,black,fill=NUred3!80}}

\makeatletter
\newcommand*{\tr}{%
	{\mathpalette\@tr{}}%
}
\newcommand*{\@tr}[2]{%
	\raisebox{\depth}{$\m@th#1\intercal$}%
}
\makeatother

\makeatletter
\let\save@mathaccent\mathaccent

\makeatother
\DeclareMathOperator*{\argmin}{arg\,min}

\newcommand*{\1}{\ensuremath{\mathbb{1}}}
\newcommand*{\La}{\ensuremath{\mathcal{L}}}
\newcommand*{\R}{\ensuremath{\mathbb{R}}}
\newcommand*{\n}{\mkern-1.5mu}

\theoremstyle{plain}
\newtheorem{theorem}{Theorem}
\newtheorem{lemma}{Lemma}
\theoremstyle{definition}
\newtheorem{remark}{Remark}
\AfterEndEnvironment{theorem}{\noindent\ignorespaces}
\AfterEndEnvironment{remark}{\noindent\ignorespaces}

\newcommand{\Mod}[1]{\ (\mathrm{mod}\ #1)}

\begin{document}
	\title{Self-Healing First-Order Distributed Optimization with Packet Loss}

	\author{Israel L. Donato Ridgley$^{1,4}$, \IEEEmembership{Member, IEEE}, \\ Randy A. Freeman$^{1,3,4}$, \IEEEmembership{Member, IEEE}, and Kevin M. Lynch$^{2,3,4}$, \IEEEmembership{Fellow, IEEE}
		\thanks{All authors are affiliated with Northwestern University, Evanston, IL 60208 USA (e-mail: israelridgley@gmail.com; freeman@northwestern.edu; kmlynch@northwestern.edu).}
		\thanks{$^1$Department of Electrical \& Computer Engineering; $^2$Department of Mechanical Engineering; $^3$Northwestern Institute on Complex Systems; $^4$Center for Robotics and Biosystems}
		\thanks{This material is based upon work supported by the National Science Foundation under Grant No.\ CMMI-2024774.}
	}
	
	\maketitle

	\begin{abstract}
		We describe SH-SVL, a parameterized family of first-order distributed optimization algorithms that enable a network of agents to collaboratively calculate a decision variable that minimizes the sum of cost functions at each agent. These algorithms are \emph{self-healing} in that their convergence to the correct optimizer can be guaranteed even if they are initialized randomly, agents join or leave the network, or local cost functions change. We also present simulation evidence that our algorithms are self-healing in the case of dropped communication packets. Our algorithms are the first single-Laplacian methods for distributed convex optimization to exhibit all of these characteristics. We achieve self-healing by sacrificing internal stability, a fundamental trade-off for single-Laplacian methods. 
	\end{abstract}
	
	\section{Introduction}
	\label{sec:intro}
	The distributed optimization problem involves a network of $n$ agents that each calculates a decision vector that minimizes a global additive objective function of the form $f(\cdot)=\sum_i f_i(\cdot)$, where $f_i$ denotes the local convex objective function known only to agent $i$. Specifically, each agent maintains a local estimate $x_i$ of the global minimizer
	\begin{equation}
		x_{\text{opt}} = \argmin_{\theta} \sum_i f_i(\theta),
	\end{equation}
	which we assume is unique.
	The agents reach consensus $x_i = x_{\text{opt}}$ by computing the gradients of their local objective functions $\nabla\n f_i(x_i)$ and passing messages along the links of the communication network.
	
	Distributed optimization problems of this form have broad application. For example, a distributed set of servers or sensors could perform a learning task (e.g., classification) using their local data without uploading the data to a central server for bandwidth, resiliency, or privacy reasons~\cite{forcangia10}. Swarms of robots can use distributed optimization to plan motions to solve the rendezvous problem \cite{rig08}. 
	
	The optimization of a collective cost function in a network setting has seen considerable interest over the last decade \cite{sunvanles20,shiqingwuyin15,lishiyan19,jak19,nedolsshi17,quli18,yuayinzhasay19p1,yuayinzhasay19p2}.
	Recently, several authors have adapted methods from control theory to study distributed optimization algorithms as linear systems in feedback with uncertainties constrained by integral quadratic constraints (IQCs) \cite{lesrecpac16,sunhules17,sunvanles20}. These works have made it possible to more easily compare the various known algorithms across general classes of cost functions and graph topologies. Furthermore, the framework has been extended to stochastic optimization algorithms by using techniques from Markov Jump Linear Systems in the centralized setting \cite{huseiran17} as well as the setting of synchronous packet loss in autonomous vehicle fleets \cite{seisen05}.
	
	The work \cite{sunvanles20} uses these techniques to describe several recent distributed gradient-tracking optimization algorithms within a common framework, then describes a new algorithm (SVL) within that framework that achieves a superior worst-case convergence rate. However, all of the algorithms considered in \cite{sunvanles20} share a common undesirable trait:
	to reach the correct solution, their states must start in a particular subspace of the overall global state space and remain in it at every time step.
	If for any reason the state trajectories 
	leave this 
	subspace (e.g., incorrect initialization, dropped packets, 
	computation errors, agents leaving the network, changes to objective functions due to continuous data collection), then the system will no longer converge to the minimizer.
	Such methods cannot automatically recover from disturbances or other faults that displace their trajectories from this subspace; in other words, they are not \textit{self-healing}. The method presented in \cite{mayulan21} is self-healing; however, the decaying gradient step size limits the rate of convergence and causes disturbances to generate increasingly long transients. 
	
	In this paper, we extend our results from dynamic average consensus estimators \cite{ridfrelyn20,kiavancorfrelynmar19} to design SH-SVL, a family of self-healing distributed optimization algorithms whose trajectories 
	need not evolve on a 
	pre-defined subspace.  
	In practice, this means that our algorithms
	can be arbitrarily initialized,
	agents can join or leave the network at will, and agents can change their objective functions as necessary, such as when they collect new data. An important consequence of the self-healing property is that our algorithm can be modified with a low-overhead packet-loss protocol which allows the algorithm to recover from lost or corrupted packets.
	
	We refer to distributed optimization algorithms that communicate one or two variables (having the same vector dimension as the decision variable $x_i$) per time step as single- and double-Laplacian methods, respectively. Examples of single-Laplacian methods are SVL and NIDS, while examples of double-Laplacian methods are uEXTRA and DIGing \cite{sunvanles20,lishiyan19,jak19,nedolsshi17,quli18}. SH-SVL is the first self-healing single-Laplacian methods for convex optimization that converge to the exact (rather than an approximate) solution with a fixed step size (see \cite{ridfrelyn20, hadvaidom16} for the specific case of average consensus). They achieve self-healing by sacrificing internal stability, a fundamental trade-off for single-Laplacian gradient tracking-methods. 
	In particular, each agent will have an internal state that grows linearly in time in steady state, but because such growth is not exponential it
	will 
	not cause any numerical issues 
	unless the algorithm runs over a long time horizon. Double-Laplacian methods can achieve both internal stability and self-healing, but they require twice as much communication per time step 
	and converge no faster than single-Laplacian methods \cite{sunvanles20,kiavancorfrelynmar19}.
 
	This paper uses the same algorithms as~\cite{ridfrelyn21}, but provides theoretical and numerical proof for convergence of SH-SVL (Algorithm~2) with packet loss. The performance of SH-SVL is explored in the setting of synchronous packet loss on arbitrary graphs and edgewise packet loss in small directed cycles. We perform a numerical study to investigate the performance when algorithm parameters are optimized for the specific level of packet loss versus no packet loss. In addition, we show that SH-SVL can recover the same performance as SVL when parameters are optimized.
	
	\section{Preliminaries and Main Results}
	
	\subsection{Notation and terminology}
	
	Let $\1_n$ be the $n$-dimensional column vector of all ones, $I_n$ be the identity matrix in $\R^{n \times n}$, and $\Pi_n = \frac{1}{n}\1 \1^\tr$ be the projection matrix onto the vector $\1_n$. We drop the subscript $n$ when the size is clear from context. We refer to the one-dimensional linear subspace of $\R^n$ spanned by the vector $\1_n$ as the \emph{consensus direction} or the \emph{consensus subspace}. We refer to the $(n-1)$-dimensional subspace of $\R^n$ associated with the projection matrix $(I_n - \Pi_n)$ as the \emph{disagreement direction} or subspace. The $i$-th column of the identity matrix is denoted $e_i$, with a size determined by context. The $i$-th row of the matrix $A$ is written $A_{r_i} = e_i^\tr A$. Given $A\in \R^{n\times n}$, we let $D(A)\in  \R^{n\times n}$ denote the diagonal matrix whose diagonal elements are the same as A.
	
	Subscripts denote the agent index whereas superscripts denote the time index. The symbol $\otimes$ represents the Kronecker product. $A^+$ indicates the Moore-Penrose inverse of $A$. Symmetric quadratic forms $x^\tr\n A x$ are written as $[\star]^\tr\n A x$ 
	to save space when $x$ is long. The local decision variables are $d$-dimensional and represented as a row vector, i.e., $x_i \in \R^{1 \times d}$, and the local gradients are a map $\nabla\n f_i : \R^{1 \times d} \rightarrow \R^{1 \times d}$. The symbol $||\cdot||$ refers to the Euclidean norm of vectors and the spectral norm of matrices. The condition number of the symmetric matrix $A$ is defined by $\operatorname{cond}(A) = \lambda_{\textnormal{max}}(A) / \lambda_{\textnormal{min}}(A)$.  The expected value operator is $\mathbb{E}[\cdot]$.
	
	We model a network of $n$ agents participating in a distributed computation as a 
	weighted digraph $\mathcal{G} = (\mathcal{V},\mathcal{E})$, where $\mathcal{V} = \{1,...,n\}$ is the set of $n$ nodes (or vertices) and $\mathcal{E}$ is the set of edges such that if $(i,j) \in \mathcal{E}$ then node $i$ can receive information from $j$. We define an ordering on $\mathcal{E}$ such that the set is in ascending order according to the receiving node index followed by the sending node index.  We make use of the \textit{weighted graph} \textit{Laplacian} $\La \in \R^{n \times n}$ associated with $\mathcal{G}$ such that $-\La_{ij}$ is the weight on edge $(i,j)\in \mathcal{E}$, $\La_{ij}=0$ when $(i,j)\not\in\mathcal{E}$ and $i\neq j$, and the diagonal elements of $\La$ are 
	$\La_{ii} = -\sum_{j\neq i} \La_{ij}$,
	so that $\La \1 = 0$. 
	We define $\sigma = ||I-\Pi-\La||$, which is a parameter related to the edge weights and the graph connectivity.
	
	
	Throughout this work we stack variables and objective functions such that
	\[x^k = \begin{bmatrix}
		x_1^k \\
		\vdots \\
		x_n^k
	\end{bmatrix} \in \R^{n\times d} \;\; \text{and} \;\;
	\nabla\n F(x^k) = \begin{bmatrix}
		\nabla\n f_1(x_1^k) \\
		\vdots \\
		\nabla\n f_n(x_n^k)
	\end{bmatrix} \in \R^{n\times d}.\]
	
	If a fixed point, $x^\star$, exists for the signal $x^k$, then we denote the error signal $\tilde{x}^k = x^k - x^\star$. 
	
	\subsection{Assumptions}
	
	\begin{enumerate}[itemsep=0.25em,label=\textbf{(A\arabic*)},%
		align=left,leftmargin=*,series=assumptions]
		
		\item Given $0< \mu \leq L$, we assume that the local gradients are sector bounded on the interval $(\mu,L)$, meaning that they satisfy the quadratic inequality 
		\begin{equation*} \hspace*{-0.2in}
			[\star]^\tr
			\begin{bmatrix}
				-2\mu LI_d & (L+\mu)I_d\\
				(L+\mu)I_d & -2I_d
			\end{bmatrix}
			\begin{bmatrix}
				(x_i-x_{\text{opt}})^\tr\\
				(\nabla\n f_i(x_i)-\nabla\n f_i(x_{\text{opt}}))^\tr
			\end{bmatrix} \geq 0
		\end{equation*}
		for all $x_i \in \R^{1\times d}$, where $x_{\text{opt}}$ satisfies $\sum_{i=1}^n \nabla\n f_i(x_{\text{opt}}) = 0$. 
		We define the condition ratio as $\kappa = L/\mu$, which captures the variation in the curvature of the objective function. \label{a:1}
		
		\item The graph $\mathcal{G}$ is strongly connected. \label{a:2}
		\item The graph $\mathcal{G}$ is weight balanced, meaning that $\1^{\n\tr}\n \La = 0$. \label{a:3}
		\item The weights of $\mathcal{G}$ are such that $\sigma = ||I-\Pi-\La|| < 1$. \label{a:4}
	\end{enumerate}
	
	\begin{remark}
		Assumption \ref{a:1} is known as a \textit{sector IQC} (for a more detailed description see \cite{lesrecpac16}). \ref{a:1} is satisfied when the local objective functions are $\mu$-strongly convex with $L$-Lipschitz continuous gradients, though the assumption itself is weaker than those conditions.
	\end{remark}
	
	\begin{remark}
		Throughout this paper we assume without loss of generality that the dimension of the local decision and state variables is $d=1$.
	\end{remark}
	
	\begin{remark}
		Under appropriate conditions on the communications network, the agents can self-balance their weights in a distributed way to satisfy \ref{a:3}; 
		for example, they can use a scalar consensus filter like push-sum (see Algorithm~12 in \cite{haddomcha18}). Additionally, agents can normalize their edge weights to enforce that their weighted in-degrees (and thus their weighted out-degrees) sum to less than one in order to satisfy \ref{a:4}.
	\end{remark}
	
	\begin{lemma}
		Given assumptions \ref{a:2} -- \ref{a:4}, the output of the Laplacian is sector bounded on $(1-\sigma,1+\sigma)$ and equivalently the signals $y$ and $\La y$ satisfy the following quadratic constraint
		
		\begin{equation} \hspace*{-0.2in} \label{eq:graph_iqc}
			[\star]^\tr
			\begin{bmatrix}
				(\sigma^2-1) (I_n-\Pi_n) & (I_n-\Pi_n)\\
				(I_n-\Pi_n) & -(I_n-\Pi_n)
			\end{bmatrix}
			\begin{bmatrix}
				y\\
				\La y
			\end{bmatrix} \geq 0
		\end{equation}
		for all $y \in \R^n$. 
	\end{lemma}
	
	\begin{proof}
		From assumptions \ref{a:2} - \ref{a:4} and the definition of the spectral norm, we have that
		
		\begin{align*}
			\sigma &= \max_y \frac{||(I-\Pi-\La)y||}{||y||}\\
			&= \max_y \frac{||(I-\Pi-\La)y||}{||(I-\Pi)y||}\\
			\sigma^2 &= \max_y \frac{y^\tr (I-\Pi-\La^\tr)(I-\Pi-\La)y}{y^\tr (I-\Pi)y}\\
			&= \max_y \frac{y^\tr (I-\Pi-\La^\tr-\La + \La^\tr \La)y}{y^\tr (I-\Pi)y} \\
			&\geq \frac{y^\tr (I-\Pi-\La^\tr-\La + \La^\tr \La)y}{y^\tr(I-\Pi) y}.
		\end{align*}
		From here, rearranging terms yields \eqref{eq:graph_iqc}.
	\end{proof}
	
	\subsection{Results}
	In the following sections we present a parameterized family of distributed, synchronous, discrete-time algorithms to be be run on each agent such that, under assumptions \ref{a:1}-\ref{a:4}, we achieve the following:
	
	\begin{description}[itemsep=0.25em]
		\item[Accurate convergence:] in the absence of disturbances or other faults, the local estimates $x_i$ converge to the optimizer $x_{\text{opt}}$ with an R-linear rate.
		\item[Self-healing:] the system state trajectories need not evolve on 
		a pre-defined
		subspace and will recover from events such as arbitrary initialization, temporary node failure, computation errors,
		or changes in local objectives.
		\item[Packet-loss protocol:] if agents are permitted to store a $d$-dimensional memory state for each of their neighbors, they can implement a packet-loss protocol that allows computations to continue in the event communication is temporarily lost. This extends the self-healing of the network to packet loss in a way that is not possible if the system state trajectories are required to evolve on a pre-defined subspace.
	\end{description}
	
	First we reproduce Algorithm 1 from our prior work \cite{ridfrelyn21} along with its relevant results for stability and convergence to aid in understanding the convergence analysis of Algorithm 2 (SH-SVL) with packet loss. We analyze Algorithm 2 (SH-SVL) with packet loss under two scenarios: synchronous packet loss for general graph topologies and edgewise packet loss where the graph must be known. We show numerical convergence rates for the later case only for small directed cycles. Numerical convergence rates are shown for both algorithms and the case of sub-optimal parameter choice during packet loss is considered.

	\section{Synthesis of Self-Healing\\ Distributed Optimization Algorithms}
	\subsection{Canonical first-order methods}
	As a motivation for our algorithms, we use the canonical form first described in \cite{sunvanles19} and later used as the SVL template \cite{sunvanles20}. When the communication graph is constant, many single-Laplacian methods 
	such as SVL, EXTRA and Exact Diffusion can be described in this form \cite{shiqingwuyin15,yuayinzhasay19p1,yuayinzhasay19p2,sunvanles19,sunvanles20}. Algorithms representable by the SVL template can also be expressed as a state space system $G$ in feedback with an uncertain and nonlinear block containing the objective function gradients $\nabla\n F(\cdot)$ and the Laplacian $\mathcal{L}$ shown in Figure \ref{fig:feedback}, where 
	\begin{gather} \label{eq:svl_template}
		G = \left[ \begin{array}{c|c|c}
			A & B_u & B_v \\
			\hline
			C_x & D_{xu} & D_{xv} \\
			\hline
			C_y & D_{yu} & D_{yv}\end{array}\right] = \left[ \begin{array}{cc|c|c}
			1 & \beta & -\alpha & -\gamma \\
			0 & 1 & 0 & -1 \\ 
			\hline
			1 & 0 & 0 & -\delta \\
			\hline
			1 & 0 & 0 & 0\end{array}\right] \otimes I_n.
	\end{gather}

    \begin{figure}
		\centering
		\begin{tikzpicture}
			\node (G) [block] {$G$};
			\node (NL) [block, below= 0.2cm of G] {${\begin{bmatrix} \nabla\n F(\cdot) & 0 \\ 0 & \mathcal{L} \end{bmatrix}}$};
			\node (off1) [left=of G] {};
			\node (inlabel) [left=0cm of off1] {${\begin{bmatrix}u^k\\v^k\end{bmatrix}}$};
			\node (off2) [right=of G] {};
			\node (outlabel) [right=0cm of off2] {${\begin{bmatrix}x^k\\y^k\end{bmatrix}}$};
   
			\draw [link] (G) -- (G-|off2) -- (off2|-NL) -- (NL);
			\draw [link] (NL) -- (NL-|off1) -- (off1|-G) -- (G);
		\end{tikzpicture}
		\caption{Distributed optimization algorithms
			represented as a feedback interconnection of an LTI system $G$ and an uncertain block containing the gradients and
			the graph Laplacian.}
		\label{fig:feedback}
	\end{figure}
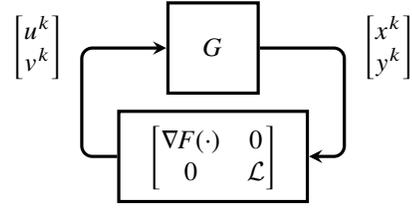%
 
	The LTI system $G$ has two states $w_1^k$ and $w_2^k$, inputs $u^k$ and $v^k$, and outputs $x^k$ and $y^k$ such that 
	\begin{gather}
		u_i^k = \nabla f_i(x_i^k), \; \; \; v_i^k = \sum_{j=1}^n \mathcal{L}_{ij}^k y_j^k.
	\end{gather}
    Additionally, the design parameters of SVL are $\alpha, \beta, \gamma,$ and $\delta$; $\alpha$ controls the influence of the gradient step, $\gamma$ and $\delta$ control the influence of the consensus step, and $\beta$ controls the coupling between the two internal states. 
 
    We would like to alert the reader to a small notational difference between our work and \cite{sunvanles20}: in this work, the variable $x$ is the input to the gradients and the variable $y$ is the input to the Laplacian (as shown in Figure \ref{fig:feedback}), whereas in \cite{sunvanles20} $y$ is the input to the gradients and $z$ is the input to the Laplacian. (We cannot use $z$ because we already use it as the frequency variable of the $z$-transform.)
 
	Algorithms described by Figure~\ref{fig:feedback} and \eqref{eq:svl_template} contain two discrete-time integrators in the LTI block $G$: one integrator is necessary to force the steady state error to zero, and the other is responsible for the agents coming to consensus. In the SVL template, the output of the graph Laplacian feeds into the integrator responsible for consensus.
	
	Algorithms representable by the SVL template, and more broadly all existing first-order methods with a single Laplacian, require that the system trajectories evolve on 
	a pre-defined
	subspace. From our work with average consensus estimators \cite{ridfrelyn20,kiavancorfrelynmar19}, we know that this drawback arises from the positional order of the Laplacian and integrator blocks. When the Laplacian feeds into the integrator, the output of the Laplacian cannot drive the integrator state away from the consensus subspace, which leads to an observable but uncontrollable mode. If the integrator state is initialized on the consensus subspace, or it is otherwise disturbed there, the estimate of the optimizer will contain an uncorrectable error.
	Switching the order of the Laplacian and integrator renders the integrator state controllable but causes it to become inherently unstable because the integrator output in the consensus direction is disconnected from the rest of the system. We exploit this trade-off to develop self-healing distributed optimization algorithms with only a single Laplacian.

	\subsection{Factorization and integrator location}
	To switch the order of the Laplacian and the integrator, we first factor an integrator out of the $G(z)$ block of Figure \ref{fig:feedback},
	\begin{gather}
		\label{eq:Gz}
		G(z) = \begin{bmatrix}
			\dfrac{-\alpha}{z-1} & -\dfrac{\delta z^2 + (\gamma - 2\delta)z + (\beta+\delta-\gamma)}{(z-1)^2}\\[8pt]
			\dfrac{-\alpha}{z-1} & -\dfrac{\gamma z + (\beta - \gamma)}{(z-1)^2}
		\end{bmatrix}\otimes I_n\\
		\label{eq:factor}
		= \begin{bmatrix}
			\dfrac{-\alpha}{z-1} & -\dfrac{z-1+\zeta}{z-1} \\[8pt]
			\dfrac{-\alpha}{z-1} & \dfrac{-\gamma z - (\beta - \gamma)}{(z-1)(\delta z + \eta - \delta)}
		\end{bmatrix}
		\begin{bmatrix}
			1 & 0 \\[8pt]
			0 & \dfrac{\delta z + \eta - \delta}{z-1}
		\end{bmatrix}\otimes I_n,
	\end{gather}
	where
	\begin{gather}
		\eta = \gamma - \delta \zeta \;\;\text{and}\;\;\zeta = \begin{cases}
			\dfrac{\beta}{\gamma} & \text{if } \delta = 0\\
			\dfrac{\gamma - \sqrt{\gamma^2 - 4\beta \delta}}{2\delta} & \text{otherwise.}
		\end{cases}
	\end{gather}
	Swapping the order of the component matrices yields our new family of algorithms (where $G_{\n e}$ replaces $G$):
	\begin{align}
		G_{\n e}(z) &=
		\begin{bmatrix}
			1 & 0 \\[8pt]
			0 & \dfrac{\delta z + \eta - \delta}{z-1}
		\end{bmatrix}
		\begin{bmatrix}
			\dfrac{-\alpha}{z-1} & -\dfrac{z-1+\zeta}{z-1} \\[8pt]
			\dfrac{-\alpha}{z-1} & \dfrac{-\gamma z - (\beta - \gamma)}{(z-1)(\delta z + \eta - \delta)}
		\end{bmatrix}\otimes I_n\notag\\
		&=
		\begin{bmatrix}
			-\alpha \dfrac{1}{z-1} & -\dfrac{z-1+\zeta}{z-1} \\[8pt]
			-\alpha \dfrac{\delta z + \eta - \delta}{(z-1)^2} & -\dfrac{\gamma z + \beta - \gamma}{(z-1)^2}
		\end{bmatrix}\otimes I_n.
	\end{align}
	\begin{figure}
		\begin{tikzpicture}
			\node (L) [block] {$\La$};
			\node (int) [block, left=of L] {${\dfrac{\delta z + \eta - \delta}{z-1}}I_n$};
			\node (off1) [left=of int] {};
			\node (off2) [right=of L] {};
			
			\draw [link] (off1) -- (int);
			\draw [link] (int) -- (L);
			\draw [link] (L) -- (off2);
		\end{tikzpicture}
		\caption{The output of the integrator now feeds into the Laplacian, converting an uncontrollable and observable mode in the original SVL template to a controllable and unobservable one.}
		\label{fig:int_order}
	\end{figure}
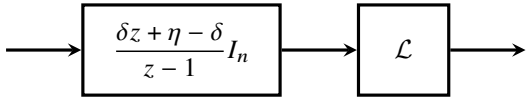%
	Now the output of the integrator feeds directly into the Laplacian, as depicted in Figure \ref{fig:int_order}. We assume that our parameter choices satisfy
	\begin{align}
		\gamma^2 &\geq 4\beta\delta
	\end{align}
	so that the zeros of $G_{\n e}$ remain real and thus the system can be implemented with real-valued signals. The corresponding distributed algorithm is described in Algorithm~\ref{alg:1}, where $w_1$ and $w_2$ are the internal states of $G_{\n e}$, and the compact state space form is 
	\begin{equation}
		G_{\n e} = \left[ \begin{array}{cc|c|c}
			1 & 0 & -\alpha & -\zeta \\
			1 & 1 & 0 & -1 \\ 
			\hline
			1 & 0 & 0 & -1 \\
			\hline
			\delta & \eta & 0 & 0\end{array}\right] \otimes I_n.
	\end{equation}
	The design parameters of Algorithm 1 (and later SH-SVL) are $\alpha, \delta, \zeta,$ and $\eta$; $\alpha$ controls the influence of the gradient step, $\zeta$ controls the influence of the consensus step, and $\delta$ and $\eta$ control the mixture of internal states being transmitted to neighbors. 
	\begin{remark}
		The factorization in \eqref{eq:factor} is not unique; we chose it because it
		leads to a method still having only two internal states per agent.
		There may be other useful factorizations.
	\end{remark}

	\begin{algorithm}[t]
		\SetAlgoLined
		\SetCommentSty{\small}
		\LinesNotNumbered
		\KwSty{Initialization:} Each agent $i\in\{1,...,n\}$ chooses $w_{1i}^0,w_{2i}^0 \in \R^{1 \times d}$ arbitrarily. $\La \in \R^{n \times n}$ is the graph Laplacian.\\
		\For{$k=0,1,2,...$}{
			\For{$i\in\{1,...,n\}$}{
				\KwSty{Local communication}\\
				$y_i^k = \delta w_{1i}^k + \eta w_{2i}^k$\\[2pt]
				$v_i^k = \sum_{j=1}^n\mathcal{L}_{ij} y_j^k$\\
				\KwSty{Local gradient computation}\\
				$x_i^k = w_{1i}^k - v_i^k$ \texttt{\scriptsize // Update the optimizer estimate.} \\[2pt]
				$u_i^k = \nabla\n f_i(x_i^k)$\\
				\KwSty{Local state update}\\
				$w_{1i}^{k+1} = w_{1i}^k - \alpha u_i^k - \zeta v_i^k$\\[2pt]
				$w_{2i}^{k+1} = w_{1i}^k + w_{2i}^k - v_i^k$
		}}
		\KwRet{$x_i$}
		\caption{\small Self-Healing Distributed Gradient Descent}
		\label{alg:1}
	\end{algorithm}
	
	\begin{remark}
		In contrast to algorithms like SVL, Algorithm 1 does not require specific initial conditions and system trajectories are not restricted to a pre-defined subspace. If agents change their local objective functions or drop out of the computation, the system does not need to be reset and the system will converge to the new minimizer. In the case of agents dropping out, the connection topology must still be strongly connected, otherwise Algorithm 1 with a weight balancer will converge to the minimizer for only a subset of the objective functions and consensus across the network will not be achieved. 
	\end{remark}
	
	\section{Stability and Convergence Rates Using IQCs}
	\subsection{Projection onto the disagreement subspace}
	As written, our family of algorithms is internally unstable. We use the projection matrix $(I-\Pi)$ to eliminate the instability from the global system without affecting $x^k$. This procedure is a centralized calculation that cannot be implemented in a distributed fashion, but it allows us to analyze the convergence properties of the distributed algorithm. 
	
	Consider the steady-state values $(w_1^\star,x^\star,u^\star,v^\star)$ and suppose $w_2^k$ contains a component in the $\1$ direction. Then that component does not affect the aforementioned values because it is an input to the Laplacian $\La$ (and lies in its nullspace); however, it grows linearly in time due to the $w_2$ update. Thus the system has an internal instability that is unobservable from the output of the bottom block in Figure~\ref{fig:feedback}. Since the component of $w_2^k$ in the consensus direction is unobservable to the variables $(w_1^k,x^k,u^k,v^k)$, we can throw it away without affecting their trajectories. Using the transformation $\hat{w}_2^k = (I - \Pi)w_2^k$, our state updates become
	\begin{align}
		\label{eq:up1}
		w_1^{k+1} &= w_1^k - \alpha u^k - \zeta v^k\\
		\label{eq:up2}
		\hat{w}_2^{k+1} &= (I-\Pi)w_1^k + (I-\Pi)\hat{w}_2^k - (I-\Pi)v^k \\
		\label{eq:up3}
		x^k &= w_1^k - v^k \\
		\label{eq:up4}
		\hat{y}^k &= \delta w_1^k + \eta \hat{w}_2^k \\
		u^k &= \nabla\n F(x^k) \\
		\label{eq:up6}
		v^k &= \mathcal{L} \hat{y}^k,
	\end{align}
	where $y^k$ was replaced with $\hat{y}^k$ in \eqref{eq:up4} and \eqref{eq:up6} to accommodate $\hat{w}_2^k$. These updates 
	lead to the state-space system
	\begin{equation}\label{eq:gp}
		G_m = \left[ \begin{array}{cc|c|c}
			I & 0 & -\alpha I & -\zeta I \\
			I-\Pi & I-\Pi & 0 & -(I-\Pi) \\ 
			\hline
			I & 0 & 0 & -I \\
			\hline
			\delta I & \eta I & 0 & 0\end{array}\right].
	\end{equation}
	
	\subsection{Existence and optimality of a fixed point}
	Now that we have eliminated the inherent instability of the global system, we can state the following about the fixed points:
	
	\begin{theorem}
		For the system described by $G_m$, there exists a unique fixed point $(w_1^\star,\hat{w}_2^\star, x^\star,\hat{y}^\star,u^\star,v^\star)$, and the fixed point has $x^\star$ in the consensus subspace such that $x_i^\star = x_{\textnormal{opt}}$ for all $i \in \{1,\dots, n\}$, i.e., the fixed point of the system is optimal.
	\end{theorem}
	
	\begin{proof}
		First, assume that the fixed point $(w_1^\star,\hat{w}_2^\star, x^\star,\hat{y}^\star,u^\star,v^\star)$ exists. To prove that the variable $x^\star$ lies in the consensus direction, we show that $(I-\Pi)x^\star=0$. From \eqref{eq:up2} and \eqref{eq:up3} we have that
		\begin{align}
			(I-\Pi)w_1^\star &= (I-\Pi)v^\star \\
			(I-\Pi)x^\star &= (I-\Pi)w_1^\star - (I-\Pi)v^\star\\
			&= 0.
		\end{align}
		Thus $x_i^\star = x_j^\star$ for all $i,j \in \{1,\dots,n\}$. Next we show that $x_i^\star = x_{\text{opt}}$. From \eqref{eq:up1} then plugging in \eqref{eq:up6}, we have
		\begin{align}
			-\alpha u^\star - \zeta v^\star &= 0 \\
			u^\star &= -\frac{\zeta}{\alpha} v^\star = -\frac{\zeta}{\alpha} \mathcal{L} \hat{y}^\star \\
			\1^\tr u^\star &= -\frac{\zeta}{\alpha} \1^\tr \mathcal{L} \hat{y}^\star \\
			\sum_{i=1}^n u_i^\star &= 0 \\
			\rightarrow \sum_{i=1}^n \nabla\n f_i(x_i^\star) &= 0\\
			\rightarrow x_i^\star &= x_{\text{opt}} \; \forall \; i \in \{1,\dots,n\}.
		\end{align}
		Thus the fixed point is optimal.
		
		Next, to construct the fixed point we define
		\begin{equation}
			\begin{aligned}
				x^\star &= \1 x_{\text{opt}}, & u^\star &= \nabla\n f(x^\star) \\
				v^\star &= -\frac{\alpha}{\zeta} u^\star, & w_1^\star &= x^\star + v^\star .
			\end{aligned}
		\end{equation}
		Then $\hat{w}_2^\star$ is the solution to the equation
		\begin{equation}
			\label{eq:w2sol}
			\zeta \eta \mathcal{L} \hat{w}_2^\star = -\alpha(I-\delta \mathcal{L})u^\star.
		\end{equation}
		Since $\hat{w}_2^k = \La^+ \La w_2^k$ (i.e., $\hat{w}_2^k$ is in the row space of $\La$), we write $\hat{w}_2^\star$ in closed form as
		\begin{equation}
			\hat{w}_2^\star = \frac{\alpha}{\zeta \eta}\La^+(\delta \La - I) u^\star.
		\end{equation}
		Finally, setting $\hat{y}^\star = \delta w_1^\star + \eta \hat{w}_2^\star$ completes the proof.
	\end{proof}
	
	\begin{remark} \label{rem:switching}
		If the graph is switching but converges in time such that the limit of the sequence of Laplacians exists, as with a weight balancer, then a solution to \eqref{eq:w2sol} still exists and an optimal fixed point can still be found. If the sequence of Laplacians does not have a limit, then each switch of Laplacian will introduce a transient to the system, after which the system will converge to the new optimal fixed point. 
        The dependence of the fixed point on the Laplacian in \eqref{eq:w2sol} is directly related to the order of the integrator block and the Laplacian block and is an inherent drawback of known self-healing methods. The fixed points of non-self-healing methods can be independent of the graph Laplacian, which enables them to converge in the presence of switching Laplacians \cite{sunvanles20}.
	\end{remark}
 
	\subsection{Convergence}
	Following the approaches in \cite{lesrecpac16,sunhules17,sunvanles20}, we prove stability using a set of linear matrix inequalities. First we split our modified system from \eqref{eq:gp} into consensus and disagreement components. We define 
	\begin{align}
		A_m &= A_p \otimes \Pi + A_q \otimes (I-\Pi) \\
		B_{mu} &= B_{pu} \otimes \Pi + B_{qu} \otimes (I-\Pi)\\
		B_{mv} &= B_{pv} \otimes \Pi + B_{qv} \otimes (I-\Pi)
	\end{align}
	\begin{equation}
		\begin{aligned}
			A_p &= \begin{bmatrix}1 & 0 \\ 0 & 0\end{bmatrix}, & B_{pu} &= \begin{bmatrix}-\alpha\\0\end{bmatrix}, & B_{pv}&=\begin{bmatrix}-\zeta \\ 0\end{bmatrix}\\
			A_q &= \begin{bmatrix}1 & 0\\ 1 & 1\end{bmatrix}, & B_{qu} &= \begin{bmatrix}
				-\alpha \\0\end{bmatrix}, & B_{qv} &=\begin{bmatrix} -\zeta\\ -1
			\end{bmatrix}.
		\end{aligned}
	\end{equation}
	We also define the matrices
	\begin{equation}
		M_0 = \begin{bmatrix}
			-2\mu L & L+\mu \\
			L+\mu & -2
		\end{bmatrix} \; \; \text{and} \; \; M_1 =
		\begin{bmatrix}
			\sigma^2-1 & 1\\1 & -1
		\end{bmatrix}.
	\end{equation}
	
	Notice that $M_0$ is associated with the sector bound on $\nabla F$ from \ref{a:1} and that $M_1$ is associated with the sector bound on $\La$ with inputs from Lemma 1.

	We now make a statement analogous to Theorem 10 in \cite{sunvanles20}.
	\begin{theorem}
		If there exists $P,Q \in \R^{2\times 2}$, $\lambda_0,\lambda_1 \in \R$, and $\rho \in (0,1)$, with $P,Q\succ 0$ and $\lambda_0,\lambda_1\geq 0$  such that
		\begin{gather}
			\label{eq:LMI1}
			[\star]^\tr
			\left[ \begin{array}{cc|c}
				P & 0 & 0\\
				0 & -\rho^2P & 0\\
				\hline
				0 & 0 & \lambda_0M_0
			\end{array}\right]
			\left[ \begin{array}{cc}
				A_p & B_{pu} \\
				I & 0 \\
				\hline
				C_{x} & D_{xu} \\
				0 & I
			\end{array}\right] \preceq 0, \\	
			[\star]^\tr
			\label{eq:LMI2}
			\arraycolsep=2.5pt
			\left[ \begin{array}{cc|c|c}
				Q & 0 & 0 & 0\\
				0 & -\rho^2Q & 0 & 0\\
				\hline
				0 & 0 & \lambda_0M_0 & 0 \\
				\hline
				0 & 0 & 0 & \lambda_1M_1 \\
			\end{array}\right]
			\left[ \begin{array}{ccc}
				A_q & B_{qu} & B_{qv} \\
				I & 0 & 0 \\
				\hline
				C_{x} & D_{xu} & D_{xv} \\
				0 & I & 0 \\
				\hline
				C_{y} & D_{yu} & D_{yv}\\
				0 & 0 & I
			\end{array}\right]\preceq 0,
		\end{gather}
		then the following is true for the trajectories of $G_m$:
		\begin{equation}\label{eq:main}
			\Bigg \| \begin{bmatrix}
				w_1^k - w_1^\star\\[.1cm] \hat{w}_2^k - \hat{w}_2^\star
			\end{bmatrix} \Bigg \|^2 \leq \operatorname{cond}(T)\rho^{2k} \Bigg \| \begin{bmatrix}
				w_1^0 - w_1^\star\\[.1cm] \hat{w}_2^0 - \hat{w}_2^\star
			\end{bmatrix} \Bigg \|^2
		\end{equation}
		for a fixed point $(w_1^\star,\hat{w}_2^\star, x^\star,\hat{y}^\star,u^\star,v^\star)$, where $T = P \otimes \Pi_n + Q \otimes (I_n - \Pi_n)$. Thus the output $x^k$ of Algorithm 1 converges to the optimizer with the linear rate $\rho$.
	\end{theorem}
	
	\begin{proof}
		Equation \eqref{eq:main} follows directly from Theorem 4 of \cite{lesrecpac16}. First, we define the quantity $\xi^k = [(w_1^k)^\tr \; (\hat{w}_2^k)^\tr]^\tr$. We then begin by taking the Kronecker product of \eqref{eq:LMI1} with $\Pi_n$ and multiplying on both sides by $[(\tilde{\xi}^k)^\tr \; (\tilde{u}^k)^\tr]$ and its transpose respectively, as well as taking the Kronecker product of \eqref{eq:LMI2} with $I_n - \Pi_n$ and multiplying on both sides by $[(\tilde{\xi}^k)^\tr \; (\tilde{u}^k)^\tr \; (\tilde{v}^k)^\tr]$ and its transpose respectively. Making use of the state space equations \eqref{eq:up1}-\eqref{eq:up6}, we get the equations
		\begin{gather}
			(\tilde{\xi}^{k+1})^\tr (P \otimes \Pi_n) \tilde{\xi}^{k+1} -\rho^2(\tilde{\xi}^{k})^\tr (P\otimes \Pi_n) \tilde{\xi}^{k} \\ \notag
			+ \lambda_0 [\star]^\tr (M_0 \otimes \Pi_n) \begin{bmatrix} \tilde{x}^k \\ \tilde{u}^k\end{bmatrix} \leq 0 \\
			(\tilde{\xi}^{k+1})^\tr (Q \otimes I_n-\Pi_n) \tilde{\xi}^{k+1} -\rho^2(\tilde{\xi}^{k})^\tr (Q\otimes I_n-\Pi_n) \tilde{\xi}^{k} \\ \notag
			+ \lambda_0 [\star]^\tr (M_0 \otimes I_n-\Pi_n) \begin{bmatrix} \tilde{x}^k \\ \tilde{u}^k\end{bmatrix} + \lambda_1 [\star]^\tr (M_1 \otimes I_n-\Pi_n) \begin{bmatrix} \tilde{\hat{y}}^k \\\tilde{v}^k\end{bmatrix} \leq 0.
		\end{gather}
		
		We then add together these equations to get
		
		\begin{gather}\label{stl_step}
			(\tilde{\xi}^{k+1})^\tr T \tilde{\xi}^{k+1} -\rho^2(\tilde{\xi}^{k})^\tr T \tilde{\xi}^{k} \\ \notag
			+ \lambda_0 [\star]^\tr (M_0 \otimes I_n) \begin{bmatrix} \tilde{x}^k \\ \tilde{u}^k\end{bmatrix} + \lambda_1 [\star]^\tr (M_1 \otimes I_n-\Pi_n) \begin{bmatrix} \tilde{\hat{y}}^k \\\tilde{v}^k\end{bmatrix} \leq 0.
		\end{gather}
		
		Notice that the terms associated with $M_0$ and $M_1$ are quadratic constraints and nonnegative, therefore we can subtract them from both sides. Next we multiply by $\rho^{-2k}$ and sum over $k$ which causes a telescoping sum and yields
		\begin{equation} \label{last_step}
			(\tilde{\xi}^{k})^\tr T \tilde{\xi}^{k} \leq \rho^{2k}(\tilde{\xi}^{0})^\tr T \tilde{\xi}^{0},
		\end{equation} 
		which leads directly to the statement of \eqref{eq:main}.
		
		Since the states of $G_m$ are converging at a linear rate $\rho$, the rest of the signals in the system (including $x^k$) converge to an optimal fixed point at the same rate. Additionally, the trajectories of $G_m$ and $G_{\n e}$ (Algorithm 1) are the same, save for $\hat{w}_2^k$ and $\hat{y}^k$, so $x^k$ in Algorithm 1 also converges to the optimizer with linear rate $\rho$. 
	\end{proof}

	\begin{remark}
		In Algorithm 1, the signal $x^k$ converges to the optimizer with linear rate $\rho$ as the internal signals $w_2^k$ and $y^k$ grow linearly. The signals $w_2^k$ and $y^k$ will not become large enough to cause overflow errors until well after the algorithm has converged to the optimizer.
	\end{remark}
	\begin{remark}
		We tested the convergence rates for our algorithm with Zames-Falb IQCs in place of sector IQCs but saw no improvement. 
	\end{remark}
	
	\begin{remark}
		The explicit dependence of LMIs \eqref{eq:LMI1} and \eqref{eq:LMI2} on $\mu$ and $L$ can be removed to facilitate comparison of objective functions based on condition ratio $\kappa$ alone. First, we define a new signal $g^k = \frac{1}{L}u^k$ and a new parameter $\alpha' = L \alpha$, such that $\alpha u^k = \alpha' g^k$. We generate new LMIs that only depend on $\kappa$ by multiplying LMIs \eqref{eq:LMI1} and \eqref{eq:LMI2} on both sides respectively by the matrices
		\begin{align*}
			\frac{1}{\sqrt{\mu L}}\begin{bmatrix} 1 & 0 & 0\\ 0 & 1 & 0 \\0 & 0 & L \end{bmatrix} \text{ and } \frac{1}{\sqrt{\mu L}}\begin{bmatrix} 1 & 0 & 0 & 0\\ 0 & 1 & 0 & 0 \\0 & 0 & L & 0 \\ 0 & 0 & 0 & 1 \end{bmatrix}.
		\end{align*}
		This corresponds to using the IQC
		\[ [\star]^\tr
		\begin{bmatrix}
			-2 & \kappa+1\\
			\kappa+1 & -2\kappa
		\end{bmatrix}
		\begin{bmatrix}
			(x_i-x_{\text{opt}})^\tr\\
			L(\nabla\n f_i(x_i)-\nabla\n f_i(x_{\text{opt}}))^\tr
		\end{bmatrix} \geq 0.\]
		If LMIs \eqref{eq:LMI1} and \eqref{eq:LMI2} are feasible then these new LMIs will be feasible as well with $P' = \frac{1}{\mu L}P$ and $Q' = \frac{1}{\mu L}Q$.
	\end{remark}
	
	\section{Robustness to Packet Loss}
	\subsection{Packet-loss protocol}
	We next give our agents some additional memory so that they can substitute previously transmitted values when a packet is lost. Each agent $i\in \{1,\dots,n\}$ maintains an edge state $r_{ij}^k$ for each $j\in \mathcal{N}_{\text{in}}(i)$ (the set of neighbors who transmit to $i$). Whenever agent $i$ receives a message from agent $j$, it updates the state $r_{ij}$ accordingly; however, if at time $k$ no message from neighbor $j$ is received, agent $i$ must estimate what would have likely been transmitted. One potential strategy is to substitute in the last message received, but because $y_j$ is growing linearly in quasi steady state, this naive strategy would ruin steady-state 
	accuracy. Instead we must account for the linear growth present in our algorithm, which we can do by analyzing the quantity $y_j^{k}-y_j^{k-1}$ at the quasi fixed point $(w_1^\star, x^\star,u^\star,v^\star)$:
	\begin{align}
		y_j^{k}-y_j^{k-1} &= \delta(w_{1j}^\star-w_{1j}^\star) + \eta(w_{2j}^k - w_{2j}^{k-1})\\
		&= \eta(w_{1j}^\star-v_j^\star)\\
		&= \eta x_j^\star \approx \eta x_i^k
	\end{align}
	Therefore, when a packet is not received by a neighbor, agent $i$ scales its estimate of the optimizer and adds it to its previously received (or estimated) message. The packet-loss protocol is summarized in Algorithm 2 (SH-SVL). By construction, the modifications included in SH-SVL will not alter the quasi fixed points of Algorithm 1 and, in the absence of dropped packets, the state trajectories of SH-SVL are equivalent to those of Algorithm 1. In the following sections, we present a similar analysis to that of the base algorithm for the simplified case of synchronous packet loss with an unknown graph and the general case of edgewise packet loss where knowledge of the graph topology must be included in the analysis.
	
	\begin{remark}
		SH-SVL can be modified to include a forgetting factor $q$. If agent $i$ does not receive a packet from neighbor $j$ in $q$ time steps, then agent $i$ assumes that the communication link has been severed and clears $r_{ij}$ from memory.
	\end{remark}
	
	\begin{algorithm}[t]
		\SetAlgoLined
		\LinesNotNumbered
		\SetCommentSty{\small}
		\KwSty{Initialization:} Each agent $i\in\{1,...,n\}$ chooses $w_{1i}^0,w_{2i}^0 \in \R^{1 \times d}$ arbitrarily. $\La \in \R^{n \times n}$ is the graph Laplacian.  All $r_{ij}$ are initialized the first time a message is received from a neighbor.\\
		\For{$k=0,1,2,...$}{
			\For{$i\in\{1,...,n\}$}{
				\KwSty{Local communication}\\
				$y_i^k = \delta w_{1i}^k + \eta w_{2i}^k$\\[2pt]
				\For{$j\in \mathcal{N}_{\text{in}}(i)\cup i$}{
					\eIf{Packet from $j$ received by $i$}{
						$r_{ij}^k = y_{j}^k$
					}
					{$r_{ij}^k = \eta x_i^{k-1} + r_{ij}^{k-1}$
				}}
				$v_i^k = \sum_{j=1}^n\mathcal{L}_{ij} r_{ij}^k$\\
				\KwSty{Local gradient computation}\\
				$x_i^k = w_{1i}^k - v_i^k$ \texttt{\scriptsize // Update the optimizer estimate.}\\[2pt] 
				$u_i^k = \nabla\n f_i(x_i^k)$\\
				\KwSty{Local state update}\\
				$w_{1i}^{k+1} = w_{1i}^k - \alpha u_i^k - \zeta v_i^k$\\[2pt]
				$w_{2i}^{k+1} = w_{1i}^k + w_{2i}^k - v_i^k$
		}}
		\Return $x_i$
		\caption{Self-Healing Distributed Gradient Descent with Packet-loss protocol (SH-SVL)}
		\label{alg:2}
	\end{algorithm}
	
	\subsection{Synchronous Loss}
	The case of synchronous packet loss is applicable to noisy channels where all communications are affected simultaneously, such as in a storm or other wide-ranging phenomena. Since all communications are affected simultaneously, the specific topology of the network is allowed to be unknown, as in the base algorithm, which allows for the analysis of algorithm performance over a range of connectivity parameters and sizes instead of specific graph topologies. At each time step $k$, we assume packets are lost in the entire network with an independent and identically distributed probability $p_\ell$. We note that the following analysis can be adapted to the case where the packet loss probability follows a Markov chain; we refer the reader to the body of work on Markov jump linear systems for more information~\cite{cosframar05}. 
	
	To begin the analysis, we first assume that each agent maintains a memory state for its own output and considers its own output lost when all packets in the network are lost. For the purposes of the centralized analysis, we assume without loss of generality that each agent $i$ maintains a state for every node in the network and stores it in a vector $r_i^k$. Thus if $l$ packets are lost consecutively at time $k>l$, we have
	\begin{gather}
		r_i^k = y^{k-l} + \eta \sum_{j=1}^l x_i^{k-j}\1 \\
		v_i^k = \La_i^\tr r_i^k = \La_i^\tr y^{k-1}.
	\end{gather}
	
	Since the $x_i$ terms are annihilated by the Laplacian, in the analysis we can replace all vectors $r_i^k$ with a single vector $\hat{r}^k$, whose update is given by:
	\begin{equation} \label{eq:r_up_sync}
		\hat{r}^{k+1} = \begin{cases}
			\delta w_1^{k+1} + \eta \hat{w}_2^{k+1} & \text{if Success} \\
			\hat{r}^{k} & \text{if Loss}.
		\end{cases}
	\end{equation}
	From Equation~\eqref{eq:r_up_sync}, it is straightforward to see that the fixed point from Theorem 1 is maintained with $\hat{r}^\star = \delta w_1^\star + \eta \hat{w}_2^\star$ and that the trajectories of $x^k$ are still unaffected by the projection of $\hat{w}_2^k$.
	 With these modification, the system matrix at each time is $G_{\iota_k} \in \{G_s,G_\ell\}$ where $G_s$ corresponds to successful packet transmission in the network (with probability $1-p_\ell$) and $G_\ell$ corresponds to synchronous packet loss in the network (with probability $p_\ell$). They are as follows:
	 
	 \begin{align*}
	 	A_{\iota_k} &= A_{\iota_kp} \otimes \Pi + A_{\iota_kq} \otimes (I-\Pi), \\
	 	B_{\iota_ku} &= B_{\iota_kpu} \otimes \Pi + B_{\iota_kqu} \otimes (I-\Pi),\\
	 	B_{\iota_kv} &= B_{\iota_kpv} \otimes \Pi + B_{\iota_kqv} \otimes (I-\Pi),
	 \end{align*}
	 \begin{equation*}
	 	\begin{alignedat}{3}
	 		A_{sp} &= \begin{bmatrix}1 & 0 & 0 \\ 0 & 0 & 0\\ \delta & 0 & 0\end{bmatrix}, &\, B_{spu} &= \begin{bmatrix}-\alpha\\0\\-\alpha\delta\end{bmatrix},  &\, B_{spv} &=\begin{bmatrix}-\zeta \\ 0\\-\delta \zeta\end{bmatrix},\\	
	 		A_{sq} &= \begin{bmatrix}1 & 0 & 0 \\ 1 & 1 & 0\\ \delta +\eta& \eta & 0\end{bmatrix}, & B_{squ} &= \begin{bmatrix}-\alpha\\0\\-\alpha\delta\end{bmatrix}, & B_{sqv} &= \begin{bmatrix}-\zeta \\ -1\\-\delta \zeta-\eta\end{bmatrix},
	 	\end{alignedat}
	 \end{equation*}
 
 	\begin{equation*}
 		\begin{alignedat}{3}
 			A_{\ell p} &= \begin{bmatrix}1 & 0 & 0 \\ 0 & 0 & 0\\ 0 & 0 & 1\end{bmatrix}, &\, B_{\ell pu} &= \begin{bmatrix}-\alpha\\0\\0\end{bmatrix}, &\, B_{\ell pv} &=\begin{bmatrix}-\zeta \\ 0\\0\end{bmatrix},\\	
 			A_{\ell q} &= \begin{bmatrix}1 & 0 & 0 \\ 1 & 1 & 0\\ 0& 0 & 1\end{bmatrix}, & B_{\ell qu} &= \begin{bmatrix}-\alpha\\0\\0\end{bmatrix}, & B_{\ell qv} &= \begin{bmatrix}-\zeta \\ -1\\0\end{bmatrix},
 		\end{alignedat}
 	\end{equation*}
 
 	\begin{alignat*}{3}
 		C_x &= \begin{bmatrix}
 			I & 0 & 0
 		\end{bmatrix}, &\quad C_y &= \begin{bmatrix}
 			0 & 0 & I
 		\end{bmatrix}, &\quad D_{xv} &= -I.
 	\end{alignat*}

	\begin{equation}\label{eq:gs}
		G_{\iota_k} = \left[ \begin{array}{c|c|c}
			A_{\iota_k} & B_{\iota_ku} & B_{\iota_kv}\\
			\hline
			C_x & 0 & D_{xv}\\
			\hline
			C_y & 0 & 0\end{array}\right].
	\end{equation}

	\begin{theorem}
		If packets are lost synchronously with an i.i.d. probability $p_\ell$ and there exists $P,Q \in \R^{3\times 3}$, $\lambda_0,\lambda_1 \in \R$, and $\rho \in (0,1)$, with $P,Q\succ 0$ and $\lambda_0,\lambda_1\geq 0$  such that
		\begin{gather}
			\label{eq:LMI_sync1}
			\sum_{o\in\{s,l\}} p_o[\star]^\tr
			\left[ \begin{array}{cc|c}
				P & 0 & 0\\
				0 & -\rho^2P & 0\\
				\hline
				0 & 0 & \lambda_0M_0
			\end{array}\right]
			\left[ \begin{array}{cc}
				A_{op} & B_{opu} \\
				I & 0 \\
				\hline
				C_{x} & 0 \\
				0 & I
			\end{array}\right] \preceq 0, \\	
			\sum_{o\in\{s,l\}} p_o [\star]^\tr
			\label{eq:LMI_sync2}
			\arraycolsep=2.5pt
			\left[ \begin{array}{cc|c|c}
				Q & 0 & 0 & 0\\
				0 & -\rho^2Q & 0 & 0\\
				\hline
				0 & 0 & \lambda_0M_0 & 0 \\
				\hline
				0 & 0 & 0 & \lambda_1M_1 \\
			\end{array}\right] \times \notag \\
			\left[ \begin{array}{ccc}
				A_{oq} & B_{oqu} & B_{oqv} \\
				I & 0 & 0 \\
				\hline
				C_x & 0 & D_{xv} \\
				0 & I & 0 \\
				\hline
				C_y & 0 & 0\\
				0 & 0 & I
			\end{array}\right]\preceq 0,
		\end{gather}
		then the following is true for the trajectories of $G_{\iota_k}$:
		\begin{equation}\label{eq:main_sync}
			\mathbb{E} \Bigg[\Bigg \| \begin{bmatrix}
				w_1^k - w_1^\star\\[.1cm] \hat{w}_2^k - \hat{w}_2^\star \\[.1cm] r^k - r^\star
			\end{bmatrix} \Bigg \|^2\Bigg] \leq \operatorname{cond}(T)\rho^{2k} \Bigg \| \begin{bmatrix}
				w_1^0 - w_1^\star\\[.1cm] \hat{w}_2^0 - \hat{w}_2^\star\\[.1cm] r^0 - r^\star
			\end{bmatrix} \Bigg \|^2
		\end{equation}
		for a fixed point $(w_1^\star,\hat{w}_2^\star,r^\star, x^\star,\hat{y}^\star,u^\star,v^\star)$, where $T = P \otimes \Pi_n + Q \otimes (I_n - \Pi_n)$. Thus the output $x^k$ of SH-SVL converges in mean square to the optimizer with the linear rate $\rho$ and additionally converges almost surely.
	\end{theorem}
	\begin{proof}
		The proof of Theorem 3 follows directly from Theorem 2 above by redefining the quantity $\xi^k = [(w_1^k)^\tr \; (\hat{w}_2^k)^\tr \; (r^k)^\tr]^\tr$ and taking expectations where appropriate. See also Theorem 1 in \cite{huseiran17}.
		 Almost sure convergence can be seen by summing both sides of Equation~\eqref{eq:main_sync} over $k$, then applying Markov's inequality and the Borel-Cantelli lemma.
	\end{proof}

	\begin{remark}
		The same system $G_{\iota_k}$ can be used in a periodic loss setting to analyze the case when multiple local computation steps are taken for every communication step. The setup follows directly from the results in \cite{fryfarsei17}; however we do not include it here because the convergence rate is slower than the setting where one computation step is performed for every communication step.
	\end{remark}

	\subsection{Edgewise Packet Loss}
	In the general case of edgewise packet loss, packets are allowed to be lost according to a Markovian process on the edges. For our analysis, we consider the case where packets are lost with an independent and identically distributed probability $p_\ell$ on each edge; however, the analysis can be extended in a straightforward manner using any Markovian packet loss process desired. 
	First, we must define some additional quantities. Let  
	\begin{gather}
		E = \begin{bmatrix}
			\La_{r_1} & 0 & \dots & 0 \\
			0 & \La_{r_2} & \dots & 0 \\
			\vdots & \vdots & \ddots & \vdots \\
			0 & 0 & \dots & \La_{r_n}
		\end{bmatrix},\; \phi^k = \begin{bmatrix}
			r_1^k \\[.1cm]
			r_2^k \\
			\vdots \\
			r_n^k
		\end{bmatrix}.
	\end{gather}

	Then the state space equations for SH-SVL are
	\begin{gather}\label{eq:pw_ss}
		w_1^{k+1} = w_1^k - \zeta E \phi^k - \alpha u^k \\
		w_2^{k+1} = w_1^k + w_2^k - E \phi^k \\
		\phi_{(i,j)}^{k+1} = \begin{cases}
			\delta w_{1j}^{k+1} + \eta w_{2j}^{k+1} & \text{if } (i,j) \text{ success} \\
			\phi_{(i,j)}^k + \eta x_i^k & \text{if } (i,j) \text{ loss}
		\end{cases} \\
		x^k = w_1^k - E\phi^k \\
		u^k = \nabla F(x).
	\end{gather}
	
	To proceed with convergence analysis, we must project the unstable state $w_2^k$ onto the disagreement subspace. The state space equations with the projection then become 
	
	\begin{gather}\label{eq:pw_proj}
		w_1^{k+1} = w_1^k - \zeta E \hat{\phi}^k - \alpha u^k \\
		\hat{w}_2^{k+1} = (I-\Pi)w_1^k + (I-\Pi)\hat{w}_2^k - (I-\Pi)E \hat{\phi}^k \\
		\hat{\phi}_{(i,j)}^{k+1} = \begin{cases}
			\delta w_{1j}^{k+1} + \eta \hat{w}_{2j}^{k+1} & \text{if } (i,j) \text{ success} \\
			\phi_{(i,j)}^k + (I-\Pi)_{r_i} x^k & \text{if } (i,j) \text{ loss}
		\end{cases} \\
		x^k = w_1^k - E\hat{\phi}^k \\
		u^k = \nabla F(x).
	\end{gather}

	Now we show that the trajectories for $x^k$ are identical for the systems defined by both Equations~\eqref{eq:pw_ss} and \eqref{eq:pw_proj}, and that the fixed point for Equation~\eqref{eq:pw_proj} is the same as in Theorem 1.
	
	\begin{theorem}
		The trajectories for $w_1^k$ and $x^k$ are unchanged from the unprojected system \eqref{eq:pw_ss} to the projected system \eqref{eq:pw_proj}.
	\end{theorem}
	\begin{proof}
		We will prove the theorem by induction. First assume that the initial conditions for $w_1^0$ are the same for the two systems. Furthermore, assume that 
		\begin{gather}
			\hat{w}_2^0 = (I-\Pi)w_2^0\\
			\hat{\phi}_{(i,j)}^0 = \phi_{(i,j)}^0 - \frac{\eta}{n}\1^\tr w_2^0.
		\end{gather}
		Since the projected system is inherently centralized, we can make these modifications to the initial conditions without loss of generality. It is clear from the initial conditions and the properties of the Laplacian that $E\phi^0 = E\hat{\phi}^0$. From the update equations we can see that $x^0$ and $w_1^1$ are the same for both systems and that $\hat{w}_2^1 = (I-\Pi)w_2^1$. Finally, looking at the update for $\hat{\phi}$, we can see that
		
		\begin{align*}
			\hat{\phi}^1_{(i,j)} &= \begin{cases}
				\delta w_{1j}^1 + \eta w_{2j}^1 - \frac{\eta}{n}\1^\tr w_2^1 & \text{if } (i,j) \text{ success} \\
				(\phi_{(i,j)}^0 + \eta x_i^0) - \frac{\eta}{n}\1^\tr (w_2^0 + x^0) & \text{if } (i,j) \text{ loss}
			\end{cases}\\
		&= \begin{cases}
			\delta w_{1j}^1 + \eta w_{2j}^1 - \frac{\eta}{n}\1^\tr w_2^1 & \text{if } (i,j) \text{ success} \\
			(\phi_{(i,j)}^0 + \eta x_i^0) - \frac{\eta}{n}\1^\tr w_2^1 & \text{if } (i,j) \text{ loss}
		\end{cases}\\
		&= \phi^1_{(i,j)} - \frac{\eta}{n}\1^\tr w_2^1.
		\end{align*}

		This completes the proof.
		
	\end{proof}
	
	Following directly from Theorem 1, the fixed point for the projected system is
	\begin{equation}
		\begin{aligned}
			x^\star &= \1 x_{\text{opt}}, & u^\star &= \nabla\n f(x^\star) \\
			w_1^\star &= x^\star + v^\star, & \hat{w}_2^\star &= \frac{\alpha}{\zeta \eta}\La^+(\delta \La - I) u^\star \\
			\hat{\phi}^\star_{(i,j)} &= \delta w_{1j}^\star + \eta \hat{w}_{2j}^\star.
		\end{aligned}
	\end{equation}

	In order for the system to be efficiently implemented, we must exploit the fact that the state $\hat{\phi}_{(i,j)}^k$ is always a success and absorb it into the state space equations. First, we define the modified quantities $\bar{E}$ and $\hat{\bar{\phi}}^k$ as follows

	\begin{gather}
		\bar{E} = \begin{bmatrix}
			\bar{\La}_{r_1} & 0 & \dots & 0 \\
			0 & \bar{\La}_{r_2} & \dots & 0 \\
			\vdots & \vdots & \ddots & \vdots \\
			0 & 0 & \dots & \bar{\La}_{r_n}
		\end{bmatrix},\; \hat{\bar{\phi}}^k = [\hat{\phi}_{(i,j)}^k]_{(i,j)\in \mathcal{E}},
	\end{gather}
	where $\bar{E} \in \mathbb{R}^{n,m}$, $\hat{\bar{\phi}}^k \in \mathbb{R}^m$, and $\bar{\La}_i$ is a row vector whose elements are $\La_{ij}$ if $(i,j) \in \mathcal{E}$. Note that the trajectories of the system have not been modified but the state space is now of size $2n+m$.
	
	
	Consider the collections 
	\begin{align}
		\omega^k &= \{\tau:\tau \in \mathcal{E} \text{ and } \tau \text{ successful at time } k\} \\
		\Omega &= \{\omega^k : k \in \mathbb{Z}_{\geq 0}\}.
	\end{align}
	Note that $|\Omega| = 2^{m}.$ Then we have that
	\begin{equation}
		\mathbb{P}(\omega^k = \omega) = p_\omega = (1-p_\ell)^{|\omega|}p_\ell^{m-|\omega|}.
	\end{equation}
	
	Now we can define the state matrices as follows:
	
	\begin{align}
		A_1 &= \begin{bmatrix}
			I - \zeta \delta D(\La) & -\zeta \eta D(\La) & -\zeta E \end{bmatrix}\\
		A_2 &= (I-\Pi)\begin{bmatrix}
			(I-\delta D(\La)) & (I-\eta D(\La)) & -E\end{bmatrix}\\
		B_1 &= -\alpha I \\
		C &= \begin{bmatrix}
			(I-\delta D(\La)) & -\eta D(\La) & -E
		\end{bmatrix}.
	\end{align}
	
	If $\tau = (i,j) \in \omega^k$ (i.e., edge $\tau$ is successful):
	\begin{align}
		A_{\omega_k\phi_\tau 1} &= e_j^\tr [\delta (I - \zeta \delta D(\La))+\eta(I-\Pi)(I-\delta D(\La))]\\
		A_{\omega_k\phi_\tau 2} &= e_j^\tr [\eta(I-\Pi)(I-\eta D(\La)-\delta \zeta \eta D(\La)]\\
		A_{\omega_k\phi_\tau \phi} &= e_j^\tr [-(\delta \zeta I + \eta(I-\Pi))E] \\
		B_{\omega_k\phi_\tau} &= -\alpha \delta e_j^\tr.
	\end{align}
	
	If $\tau = (i,j) \in \mathcal{E}\setminus\omega^k$ (i.e., edge $\tau$ is unsuccessful):
	\begin{align}
		A_{\omega_k\phi_\tau 1} &= e_i^\tr \eta(I-\Pi)(I+\delta D(\La)) \\
		A_{\omega_k\phi_\tau 2} &= e_i^\tr \eta^2(I-\Pi)D(\La) \\
		A_{\omega_k\phi_\tau \phi} &= e_\tau^\tr + e_i^\tr \eta(I-\Pi)E \\
		B_{\omega_k\phi_\tau} &= 0.
	\end{align}
 	Finally, 
	\begin{align}
		A_{\omega_k\phi} &= \begin{bmatrix}
			A_{\omega_k\phi_\tau 1} & A_{\omega_k\phi_\tau 2} & A_{\omega_k\phi_\tau \phi}
		\end{bmatrix}_{\tau\in \mathcal{E}} \\
		B_{\omega_k\phi} &= [B_{\omega_k\phi_{\tau}}]_{\tau\in \mathcal{E}}
	\end{align}
	
	\begin{equation}\label{eq:gpw}
		G_{\omega_k} = \left[ \begin{array}{c|c}
			A_1 & B_1 \\
			A_2 & 0 \\
			A_{\omega_k\phi} & B_{\omega_k\phi} \\ 
			\hline
			C & 0
			\end{array}\right].
	\end{equation}

	\begin{theorem}
		If packets are lost edgewise with an i.i.d. probability $p_\ell$ and there exists $P \in \R^{2n+m\times 2n+m}$, $\lambda_0\in \R$, and $\rho \in (0,1)$, with $P\succ 0$ and $\lambda_0\geq 0$ such that
		\begin{gather}
			\label{eq:LMI_ew1}
			\sum_{\omega\in\Omega} p_\omega[\star]^\tr
			\left[ \begin{array}{cc|c}
				P & 0 & 0\\
				0 & -\rho^2P & 0\\
				\hline
				0 & 0 & \lambda_0M_0
			\end{array}\right]
			\left[ \begin{array}{cc}
				A_{\omega} & B_{\omega} \\
				I & 0 \\
				\hline
				C & 0 \\
				0 & I
			\end{array}\right] \preceq 0,
		\end{gather}
		then the following is true for the trajectories of $G_{\omega_k}$:
		\begin{equation}\label{eq:main_ew}
			\mathbb{E} \Bigg[\Bigg \| \begin{bmatrix}
				w_1^k - w_1^\star\\[.1cm] \hat{w}_2^k - \hat{w}_2^\star \\[.1cm] \hat{\bar{\phi}}^k - \hat{\bar{\phi}}^\star
			\end{bmatrix} \Bigg \|^2\Bigg] \leq \operatorname{cond}(P)\rho^{2k} \Bigg \| \begin{bmatrix}
				w_1^0 - w_1^\star\\[.1cm] \hat{w}_2^0 - \hat{w}_2^\star\\[.1cm] \hat{\bar{\phi}}^0 - \hat{\bar{\phi}}^\star
			\end{bmatrix} \Bigg \|^2
		\end{equation}
		for a fixed point $(w_1^\star,\hat{w}_2^\star,\hat{\bar{\phi}}^\star, x^\star,u^\star)$. Thus the output $x^k$ of SH-SVL converges almost surely to the optimizer with the linear rate $\rho$.
	\end{theorem}
	\begin{proof}
		The result follows directly from Theorem 3.
	\end{proof}

	\section{Numerical Results}
	The linear matrix inequalities provided in Theorems 2, 3, and 5 must be solved numerically in order to certify a worst-case convergence rate $\rho$. To do so, the associated semidefinite program was implemented in the Julia programming language using JuMP and solved with MOSEK \cite{mosek}. Given a set of parameters, a bisection search was used to find $\rho$. The Nelder-Mead algorithm (from the Optim library \cite{optim}) was used to find algorithm parameters ($\alpha, \delta,\zeta,\eta$) that minimized $\rho$. We found that the Nelder-Mead algorithm needed to be initialized with parameters that certified a converging algorithm to find useful parameters and that it could effectively be warm started by providing near-optimal parameters. The worst-case convergence rate of our algorithm is subject to the same lower bound, $\rho\geq \max ((\kappa-1)/(\kappa+1),\sigma) $, found in \cite{sunvanles20}.

    \subsection{Convergence Rates}
	Figure \ref{fig:sync_loss} and \ref{fig:sync_loss20} show the performance of SH-SVL in the synchronous loss setting over a range of graph connectivity and for multiple rates of packet loss, including the case with no packet loss where Algorithm 1 and 2 are equivalent, for condition ratios of $\kappa = 10$ and $\kappa=20$ respectively. The figures also include the performance of SVL and the theoretical lower bound \cite{sunvanles20} on performance for reference. Notably, in the lossless setting with parameters optimized for a given $\kappa$ and $\sigma$, SH-SVL marginally outperforms the SVL algorithm in the worst case. This could be due to the fact that the SVL algorithm does not optimize for $\rho$ directly. Also of note is that when the parameters are optimized for the specific level of packet loss, SH-SVL converges even in the presence of extraordinarily high packet loss on graphs with very low connectivity (the value of $\sigma$ is high).
	
	\begin{figure}
		\includegraphics[width=0.5\textwidth]{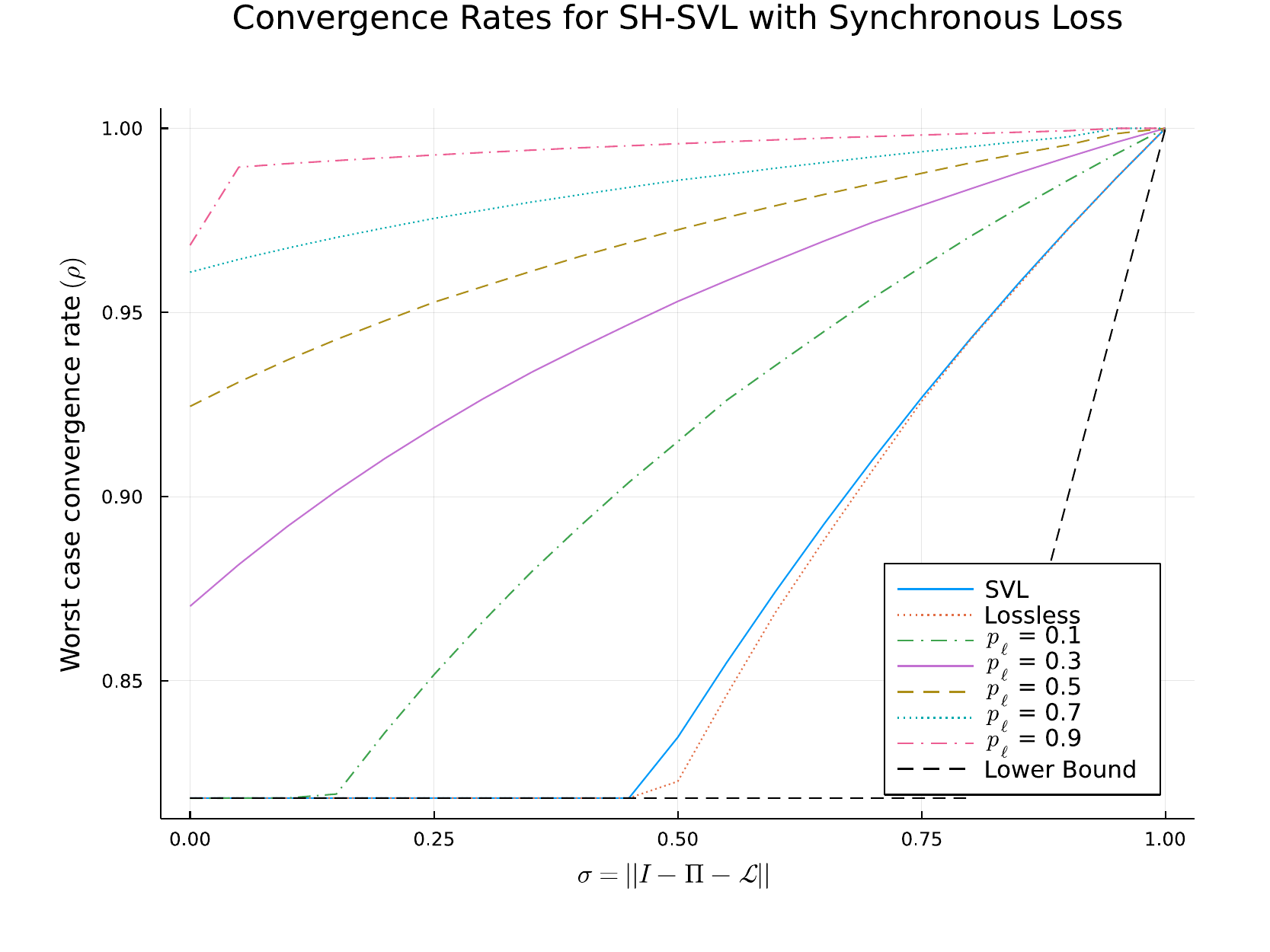}
		\caption{Performance of SH-SVL in the synchronous packet loss setting for a range of independent packet loss chances. SVL included for reference. The objective function condition ratio is $\kappa = 10$.}
		\label{fig:sync_loss}
	\end{figure}
	
	\begin{figure}
		\includegraphics[width=0.5\textwidth]{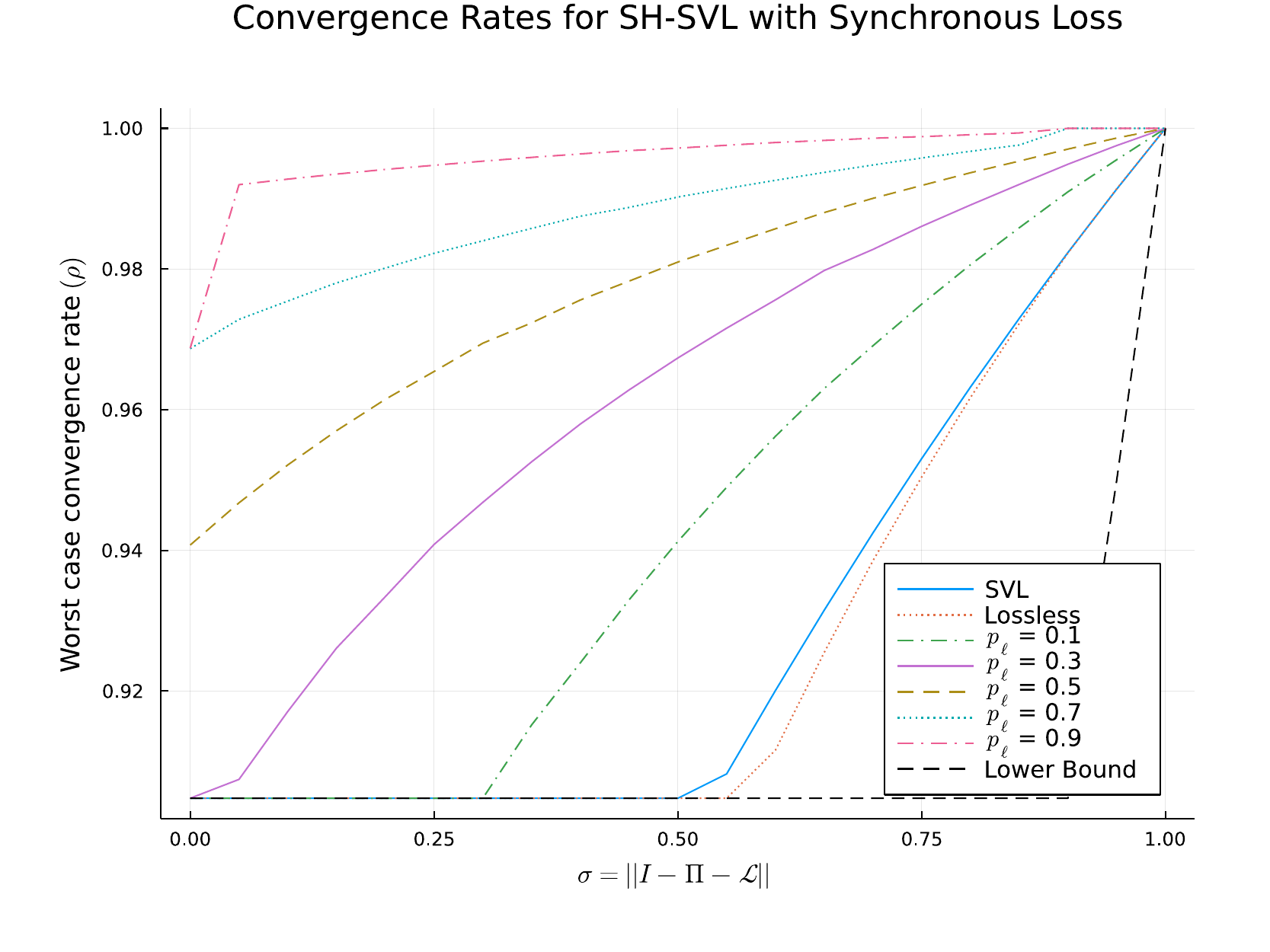}
		\caption{Performance of SH-SVL in the synchronous packet loss setting for a range of independent packet loss chances. SVL included for reference. The objective function condition ratio is $\kappa = 20$.}
		\label{fig:sync_loss20}
	\end{figure}
	
	However, if the rate of packet loss is higher than the parameters are optimized for, the performance of SH-SVL can be very fragile. Figure \ref{fig:1.0_loss} depicts the results when the optimal parameters for the lossless case are used in the lossy case. Interestingly, if the packet loss rate is lower than that which the parameters are optimized for, the worst case performance remains the same. Figure \ref{fig:0.5_loss} depicts the results when optimal parameters for $p_\ell=0.5$ are used. This would imply that system designers could decide the tradeoff between typical operating convergence speeds and the maximum allowable packet loss in the system.
	
	\begin{figure}
		\includegraphics[width=0.5\textwidth]{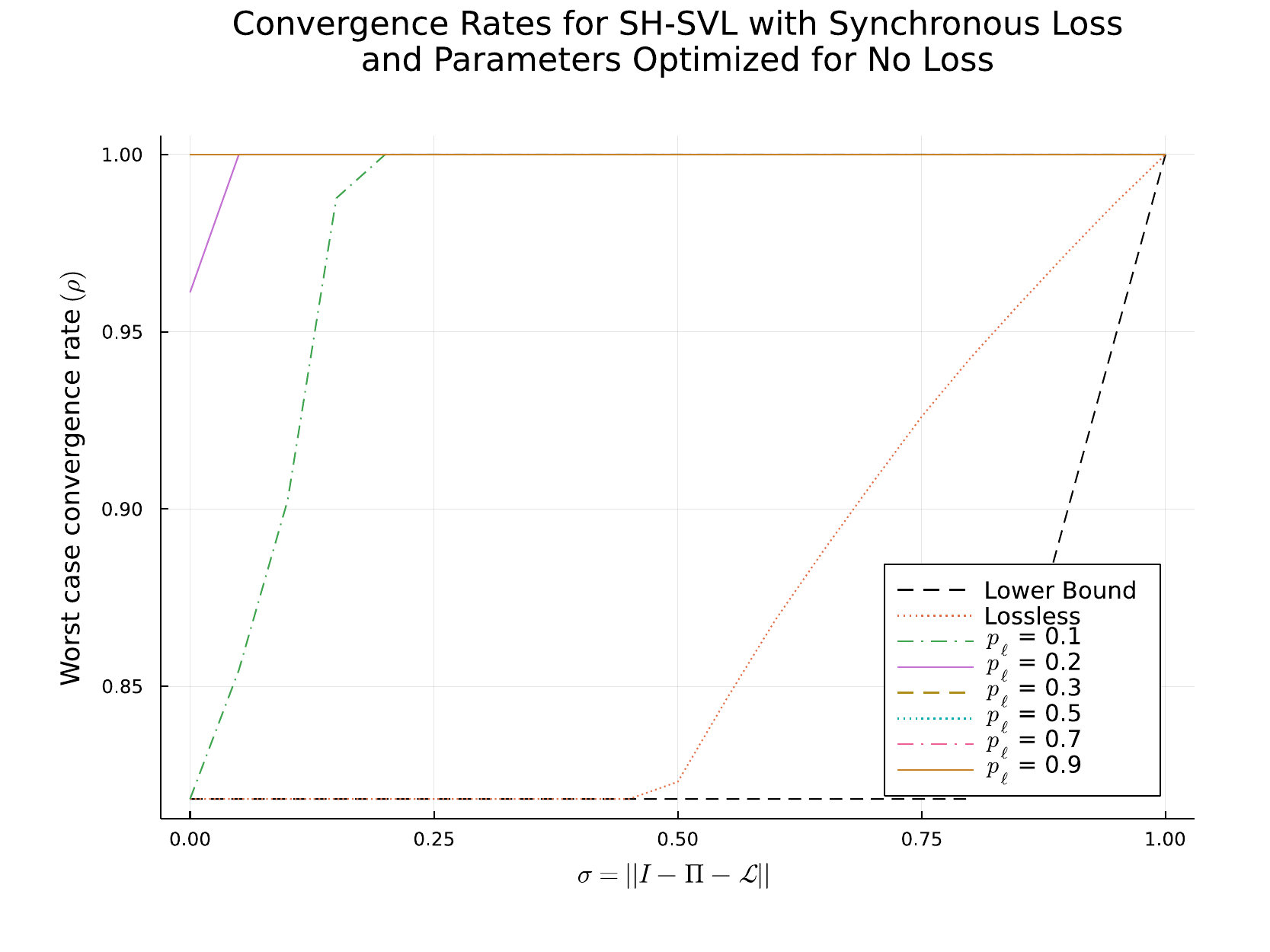}
		\caption{Performance of SH-SVL in the synchronous packet loss setting for a range of independent packet loss chances. Parameters are optimized for $\kappa=10$ and given $\sigma$ value but no packet loss.}
		\label{fig:1.0_loss}
	\end{figure}
	
	\begin{figure}
		\includegraphics[width=0.5\textwidth]{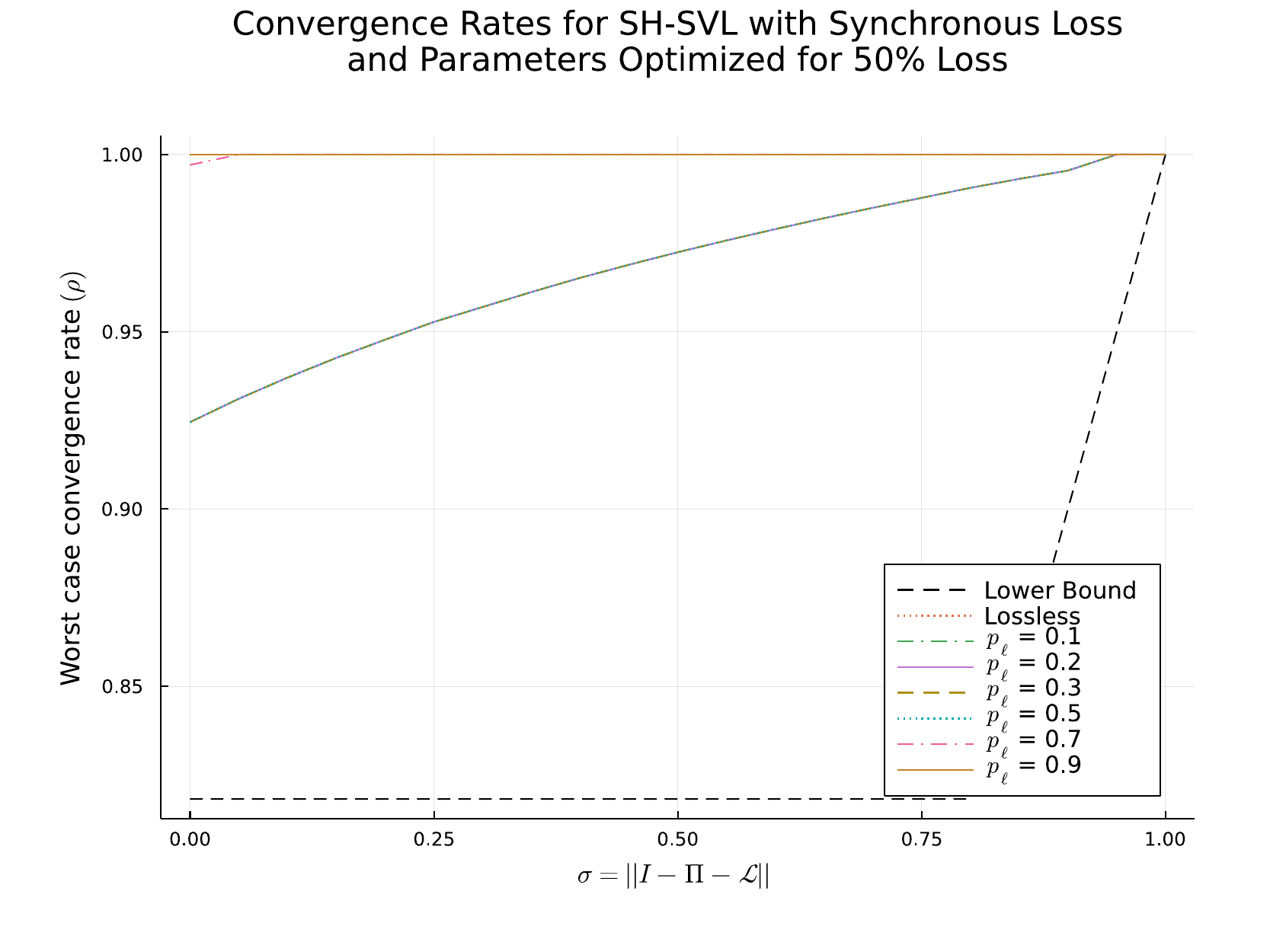}
		\caption{Performance of SH-SVL in the synchronous packet loss setting for a range of independent packet loss chances. Parameters are optimized for $\kappa=10$, $p_\ell = 0.5$, and the given $\sigma$ value.}
		\label{fig:0.5_loss}
	\end{figure}

	Figure \ref{fig:ew_loss} shows the performance of SH-SVL in both the edgewise and synchronous loss settings over the range of packet loss probabilities and for different sizes of directed cycles. The theoretical lower bound is included for reference. Directed cycles were chosen as the graph topology to demonstrate the performance of SH-SVL because they have the minimum number of edges for the graph to be connected, which makes the size of the semidefinite program smaller and also gives them a very low connectivity which makes them serve as a rough lower bound on the performance for a given graph size. Only sizes up to $n=5$ are shown due to the fact that finding optimal parameters using Nelder-Mead for such large problems starts to take a prohibitively long time. The 3-, 4-, and 5-node directed cycles have edge weights of $0.9$ and $\sigma =  [0.854, 0.906, 0.936]$ respectively. From Figures~\ref{fig:sync_loss} and \ref{fig:ew_loss}, it is clear that the LMI from Theorem 6 yields a faster convergence rate than the LMIs from Theorem 3 for the same values of $\kappa$ and $p_\ell$. This is entirely due to the fact that the LMIs in Theorem 6 are less conservative because the Laplacian is absorbed into the linear system whereas in Theorem 3, the Laplacian is parameterized. Similar LMIs to Theorem 6 can be derived for the synchronous loss case (results plotted as dashed lines in Figure \ref{fig:ew_loss}); the results suggest that the synchronous loss case is a good upper bound on performance for the edgewise loss setting.
	
	\begin{figure}[h]
		\includegraphics[width=0.5\textwidth]{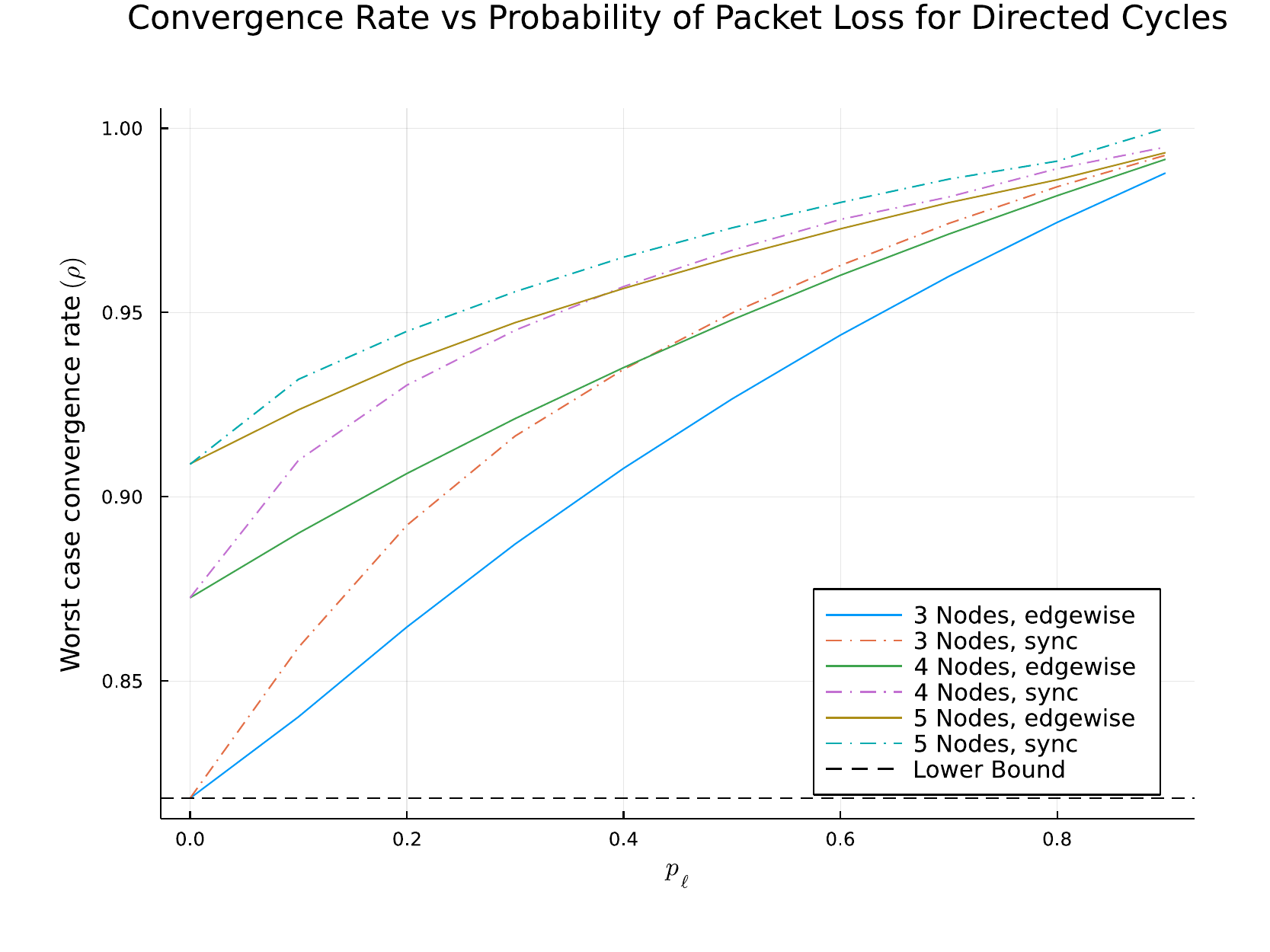}
		\caption{Performance of SH-SVL in the edgewise and synchronous packet loss setting on directed cycles with the Laplacian absorbed into the linear dynamics. Edgewise loss resultst are solid lines while synchronous loss results are in dashed lines. $\kappa=10$.}
		\label{fig:ew_loss}
	\end{figure}
	
	\subsection{Classification example}
	
	To test the performance of our algorithm under packet loss, we solved a classification problem using the COSMO chip dataset \cite{cosmo}. The problem setup is as follows: a network of $n$ agents would like to collaboratively compute a binary classifier that identifies whether or not a computer chip will pass a quality assurance test using data gathered independently by each agent in the network; furthermore, they would like to do so in a distributed fashion without sharing their datasets. In order to simulate this problem, we divide the COSMO chip dataset into $n$ local subsets and denote agent $i$'s set of local data indices as $S_i$. To improve the performance of their classifier, the agents employ a polynomial embedding where each 2-dimensional data point $d_j = [d_{j1}, d_{j2}]$ is embedded in a 28-dimensional space given by \[M(d_j) = [1, d_{j1}, d_{j2}, d_{j1}^2, d_{j1} d_{j2}, d_{j2}^2, d_{j1}^3, \dots, d_{j1}d_{j2}^5, d_{j2}^6].\]
	The agents use the logistic loss function with $L_2$-regularization, yielding local cost functions given by \begin{equation}\label{eq:loss}
		f_i(x_i) = \sum_{j\in S_i} \log(1+e^{-l_jx_i^\tr M(d_j)}) + \frac{1}{n}||x_i||^2,
	\end{equation}
	where $l_j\in \{-1,1\}$ is the label of data point $j$, and each agent computes the label of unseen data by using the operation $l = \operatorname{sgn}(x_i^\tr M(d))$. Using the cost in \eqref{eq:loss}, the corresponding sector bound $(\mu,L)$ is approximated as $\mu = \frac{2}{n}$ and
	\begin{gather}
		L_i \leq \Big\|\frac{2}{n}I + \frac{1}{4}M_i^\tr M_i\Big\|, \; \;
		L = \max_i L_i,
	\end{gather}
	where the rows of $M_i$ are $M(d_j)$ for $j\in S_i$.
	The agents' connection topology is described by an $n=7$ node directed ring lattice, shown in Figure \ref{fig:network}, such that $(i,j)\in \mathcal{E}$ when $j \in \{i+1,i+3,i+5\}\mod n$. All edge weights in the graph are set to 1/4 and $\sigma = ||I-\Pi-\La|| = 0.562$.
	
	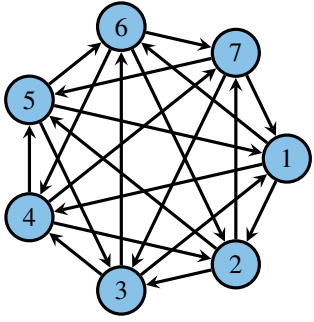
\begin{figure}
		\centering
		\begin{tikzpicture}
			\node (0) {};
			\node (1) at ($(0)+(0:1.8)$) [vertex] {$1$};
			\node (2) at ($(0)+(-51.43:1.8)$) [vertex] {$2$};
			\node (3) at ($(0)+(-102.86:1.8)$) [vertex] {$3$};
			\node (4) at ($(0)+(-154.29:1.8)$) [vertex] {$4$};
			\node (5) at ($(0)+(-205.71:1.8)$) [vertex] {$5$};
			\node (6) at ($(0)+(-257.14:1.8)$) [vertex] {$6$};
			\node (7) at ($(0)+(-308.57:1.8)$) [vertex] {$7$};

			\draw [link] (1) -- (2);
			\draw [link] (1) -- (4);
			\draw [link] (1) -- (6);
			\draw [link] (2) -- (3);
			\draw [link] (2) -- (5);
			\draw [link] (2) -- (7);
			\draw [link] (3) -- (4);
			\draw [link] (3) -- (6);
			\draw [link] (3) -- (1);
			\draw [link] (4) -- (5);
			\draw [link] (4) -- (7);
			\draw [link] (4) -- (2);
			\draw [link] (5) -- (6);
			\draw [link] (5) -- (1);
			\draw [link] (5) -- (3);
			\draw [link] (6) -- (7);
			\draw [link] (6) -- (2);
			\draw [link] (6) -- (4);
			\draw [link] (7) -- (1);
			\draw [link] (7) -- (3);
			\draw [link] (7) -- (5);
		\end{tikzpicture}
		\caption{The directed network topology for the classification example. All edge weights are 1/4.}
		\label{fig:network}
	\end{figure}
	
	Using $(\mu,L)$ and $\sigma$, we computed the optimal parameters for our algorithm using Nelder-Mead assuming no packet loss and computed the SVL parameters as detailed in \cite{sunvanles20}. We then simulated both SH-SVL and SVL with and without packet loss and took the maximum error between the distributed algorithms and a centralized solution found using Convex.jl and MOSEK. We ran SH-SVL using random initial conditions on the interval $[0,1]$ and the SVL algorithm using zero initial conditions. For the packet loss run of SVL,
	we held the previous message on each edge
	so that the fixed points would be unaffected. At each time step, packets had a 30\% chance of being lost, independent of each other and time. The results of these simulations are shown in Figure \ref{fig:packet_loss}. In this scenario, SH-SVL has similar performance as SVL and SH-SVL with lossy channels still converges to the optimum at a similar rate as SH-SVL with lossless channels, despite the high rate of packet loss that causes SVL to converge with high error. The worst-case convergence rate for both SH-SVL and SVL is $\rho = 0.9544$ and the actual rates in this scenario are $0.9514$ and $0.9455$ respectively. With a packet-loss rate of $30\%$ and parameters optimized for lossless channels, the worst-case convergence rate for SH-SVL is $\rho=1$; however, the actual rate from Figure \ref{fig:packet_loss} is $0.9632$, demonstrating the fact that the worst-case results for packet loss can be very conservative and that the optimal parameters may not be as fragile as Figure \ref{fig:1.0_loss} suggests.
	\begin{figure}[h]
		\includegraphics[width=0.5\textwidth]{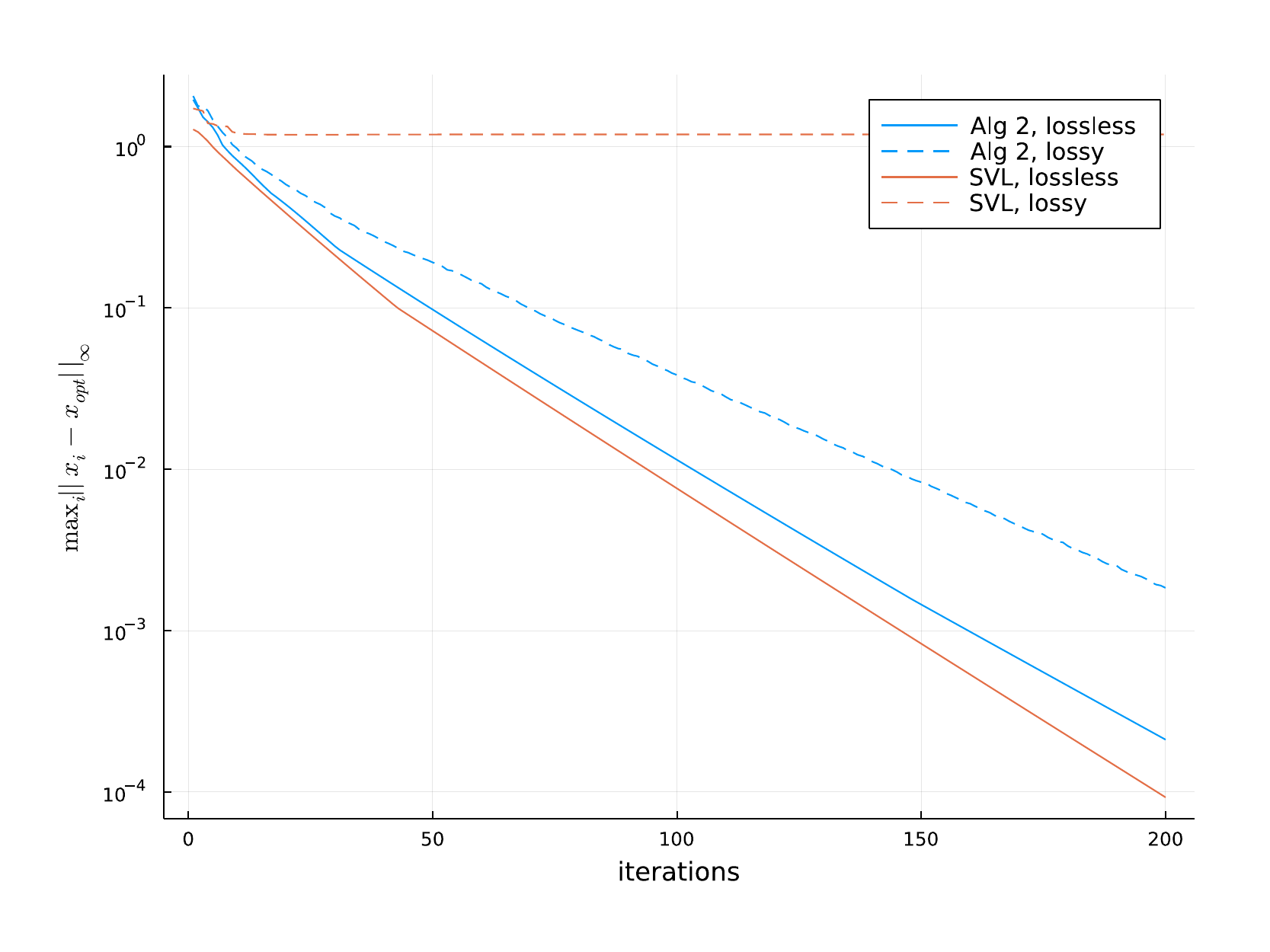}
		\caption{Simulation of SH-SVL and SVL in lossless as well as lossy channels. The lossy channels are modeled with an independent 30\% packet loss. Error is the maximum error.}
		\label{fig:packet_loss}
	\end{figure}
	\section{Summary and Future Work}
	In this paper, we extended the analysis of our of our self-healing algorithms to the case where packets are lost by constructing linear matrix inequalities for both the case of synchronous packet loss and general edgewise packet loss. We performed a numerical study which certified that SH-SVL will converge even in the case of extremely high packet loss on networks with poor connectivity. 
	
	
	In the future, we plan to extend our synthesis and analysis techniques to the more general cases of asynchronous updates, nonsmooth objectives, and nonconvex objectives.
	

	
	\renewcommand*{\bibfont}{\footnotesize}
	\printbibliography

    \vskip 0pt plus -1fil
    \begin{IEEEbiography}[{\includegraphics[width=1in,height=1.25in,clip,keepaspectratio]{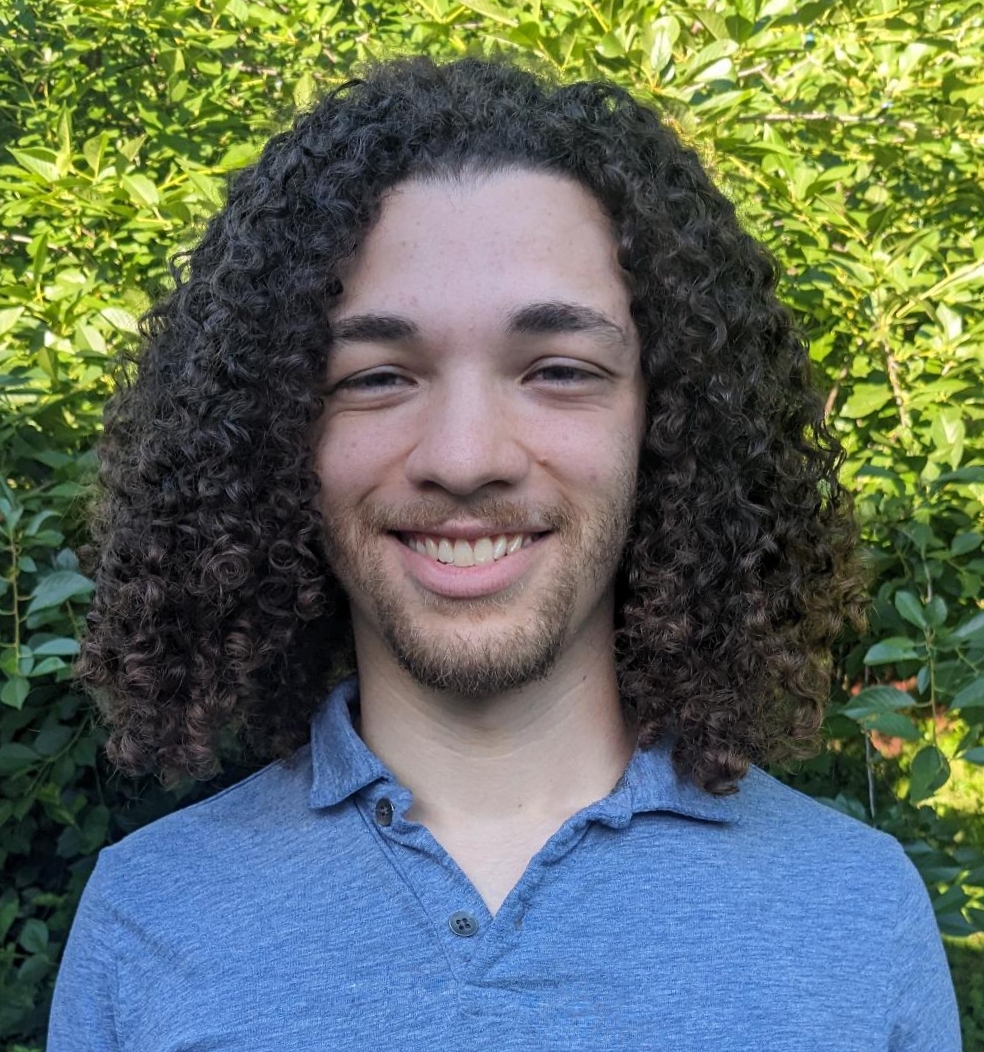}}]{Israel L. Donato Ridgley}%
  Israel Donato Ridgley is a Ph.D. candidate in Electrical Engineering at Northwestern University (Evanston, Illinois) expected to graduate in Summer 2023. He is an affiliate of the Center for Robotics and Biosystems. His research interests include distributed control and optimization, mulit-agent systems, and privacy in multi-agent systems. 
\end{IEEEbiography}

    \begin{IEEEbiography}[{\includegraphics[width=1in,height=1.25in,clip,keepaspectratio]{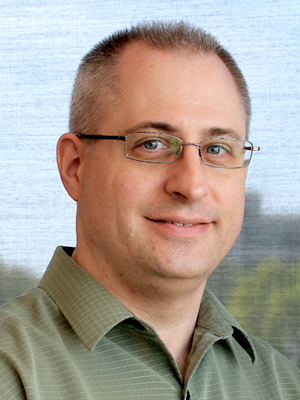}}]{Randy A. Freeman}%
  Randy Freeman received a Ph.D. in Electrical Engineering from the
University of California at Santa Barbara in 1995.  Since then he has
been a faculty member at Northwestern University (Evanston, Illinois),
where he is currently Professor of Electrical and Computer
Engineering.  He is a member of the Northwestern Institute on Complex
Systems and an affiliate of the Center for Robotics and Biosystems.
His research interests include nonlinear systems and control,
distributed control and optimization, multi-agent systems, robust
control, and optimal control.
\end{IEEEbiography}

    \begin{IEEEbiography}[{\includegraphics[width=1in,height=1.25in,clip,keepaspectratio]{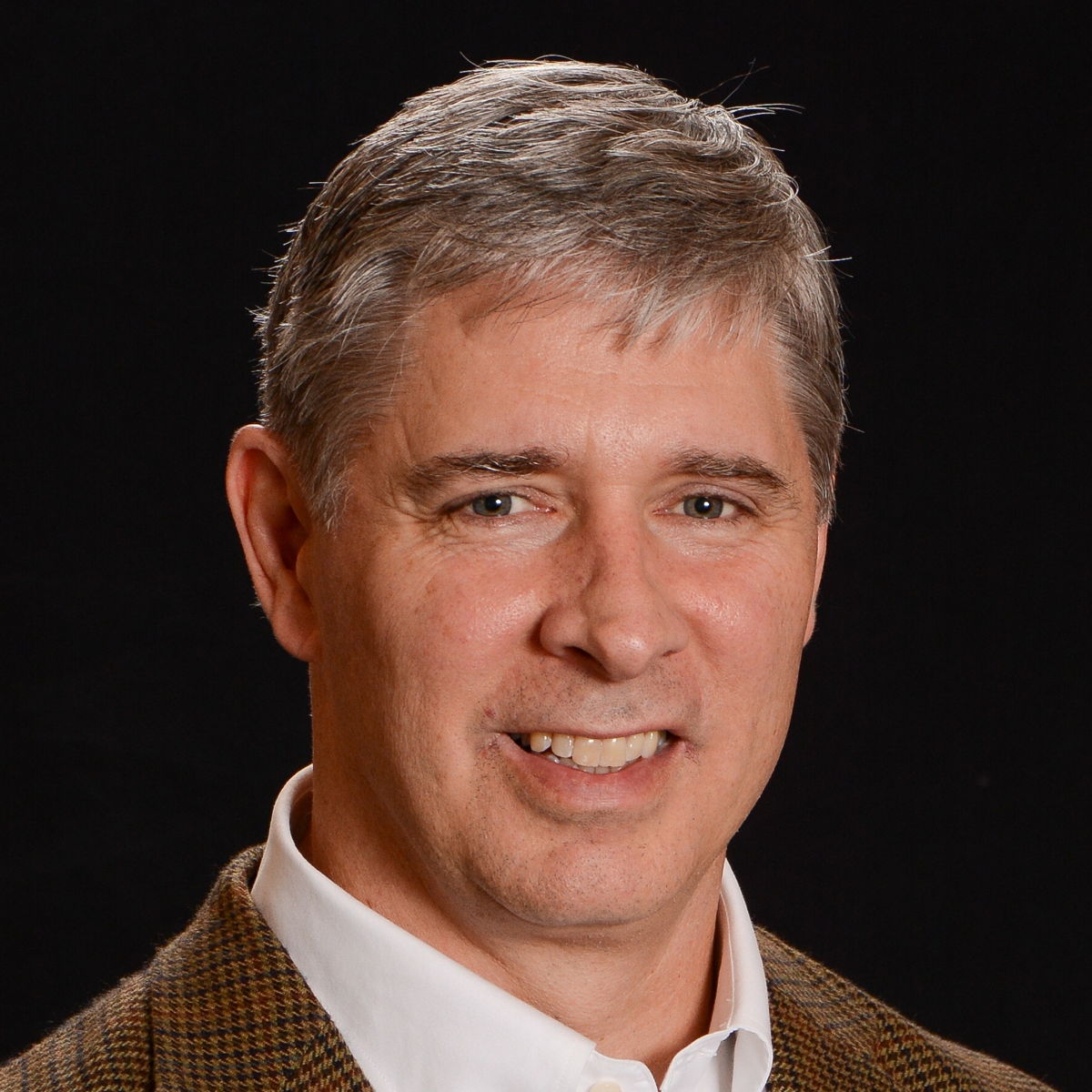}}]{Kevin M. Lynch}%
 (S’90–M’96–SM’05–F’10) received the B.S.E. degree in electrical engineering
from Princeton University, Princeton, NJ, USA, and the Ph.D. degree in robotics from Carnegie Mellon University, Pittsburgh, PA, USA.

He is a professor of mechanical engineering at Northwestern University, where he directs the Center for Robotics and Biosystems and is a member of the Northwestern Institute on Complex Systems. He is a coauthor of the textbooks \emph{Principles of Robot Motion} (Cambridge, MA, USA: MIT Press, 2005) and \emph{Modern Robotics: Mechanics, Planning, and Control} (Cambridge, U.K.: Cambridge Univ. Press, 2017) and the associated online courses and videos. His research interests include robot manipulation and locomotion, self-organizing multiagent systems, and physically collaborative human–robot systems.  
\end{IEEEbiography}
	
\end{document}